\documentclass[11pt]{amsart}
\usepackage{latexsym,color,amsmath,systeme, amsthm,amssymb,amscd,amsfonts}
\usepackage{graphicx}

\usepackage{hyperref}
\usepackage{pgf,tikz,pgfplots, pgfplotstable}
\pgfplotsset{compat=1.14}
\usepackage{float}
\usepackage{mathrsfs}
\usetikzlibrary{arrows}
\usepackage[style=english]{csquotes}
\usepackage{enumerate}
\usepackage{lscape}
\usepackage[width=\textwidth]{caption}
\usepackage{subcaption}
\usepackage{amsmath}
\usepackage{amsfonts}
\usepackage{amssymb}
\usepackage{tkz-tab}
\usepackage{siunitx}
\usetikzlibrary{angles, quotes}

\usetikzlibrary{decorations.pathmorphing}
\tikzset{discont/.style={decoration={zigzag,segment length=12pt, amplitude=4pt},decorate}}
\pgfplotsset{compat=1.11}
\usetikzlibrary{intersections,fillbetween}
 \usetikzlibrary{decorations.text}
\usetikzlibrary{shapes,arrows,intersections, backgrounds}
\usetikzlibrary{scopes,svg.path,shapes.geometric,shadows}
\DeclareCaptionSubType*[arabic]{figure}
\DeclareUnicodeCharacter{2212}{\textendash}
\DeclareUnicodeCharacter{2217}{\textendash}
\DeclareUnicodeCharacter{2032}{\textendash}

\tikzset{
	v1/.style={line width=.5pt,blue!33!black},
	v2/.style={line width=.5pt,blue!66!black},
	v3/.style={line width=.5pt,blue!33},
	v4/.style={line width=.5pt,blue!66},
	v5/.style={line width=.5pt,black}
}
\newcommand{\boundellipse}[3]
{(#1) ellipse (#2 and #3)
}
\definecolor{dred}{HTML}{C11B17}
\definecolor{dgreen}{HTML}{41A317}
\definecolor{dblue}{HTML}{00008B}
\definecolor{niceblue}{HTML}{00008B}
\definecolor{brilliantrose}{rgb}{1.0, 0.33, 0.64}
\definecolor{gold}{HTML}{ffd700}
\definecolor{lgold}{HTML}{ffe140}

\newcommand{\cemph}[1]{\textcolor{dred}{\emph{#1}}}

\newenvironment{smallpmatrix}{\left( \begin{smallmatrix}}{\end{smallmatrix} \right)}
\newcommand{\smallmat}[1]{\begin{smallpmatrix} #1 \end{smallpmatrix}}

\newcommand{\R}{\mathbb R}
\newcommand{\K}{\mathcal K}

\newcommand{\B}{\mathbb{B}}

\newcommand{\wrt}{w.r.t.}

\newcommand{\conv}{\mathrm{conv}}
\newcommand{\bd}{\mathrm{bd}}
\newcommand{\ext}{\mathrm{ext}}
\newcommand{\pos}{\mathrm{pos}}
\newcommand{\aff}{\mathrm{aff}}

\newcommand{\inter}{\mathrm{int}}
\newcommand{\lin}{\mathrm{lin}}

\newcommand{\HH}{\mathbb{H}}

\newcommand{\GH}{{\mathbb{GH}}}

\definecolor{zzttqq}{rgb}{0.6,0.2,0}
			\definecolor{ccqqqq}{rgb}{0.8,0,0}
\definecolor{ffvvqq}{rgb}{1,0.3333333333333333,0}
\definecolor{zzffqq}{rgb}{0.6,1,0}
\definecolor{qqwuqq}{rgb}{0,0.39215686274509803,0}
\definecolor{ffzzqq}{rgb}{1,0.6,0}
\definecolor{ffqqqq}{rgb}{1,0,0}
\definecolor{zzttqq}{rgb}{0.6,0.2,0}
\definecolor{uuuuuu}{rgb}{0.26666666666666666,0.26666666666666666,0.26666666666666666}

\newtheorem{thm}{Theorem}[section]
\newtheorem{lemma}[thm]{Lemma}
\newtheorem{remark}[thm]{Remark}
\newtheorem{proposition}[thm]{Proposition}
\newtheorem{cor}[thm]{Corollary}

\newtheorem{example}[thm]{Example}

\usetikzlibrary{calc,intersections}
\begin{document}

\title[Inequalities relating symmetrizations of convex bodies]{From inequalities relating symmetrizations of \\ convex bodies to the Diameter-width ratio for \\ complete and pseudo-complete convex sets}


\author[R. Brandenberg]{Ren\'e Brandenberg}
\author[K. von Dichter]{Katherina von Dichter}
\author[B. Gonz\'alez Merino]{Bernardo Gonz\'alez Merino}

\begin{abstract}
For a Minkowski centered convex compact set $K$ we define $\alpha(K)$ to be the smallest possible factor  to cover $K \cap (-K)$ by a rescalation of $\conv(K\cup (-K))$ and give a complete description of the possible values of $\alpha(K)$ in the planar case in dependence of the Minkowski asymmetry of $K$. As a side product, we show that, if the asymmetry of $K$ is greater than the golden ratio, the boundary of $K$ intersects the boundary of its negative $-K$ always in exactly 6 points.
As an application, we derive bounds for the diameter-width-ratio for pseudo-complete and complete sets, again in dependence of the Minkowski asymmetry of the convex bodies, 
tightening those depending solely on the dimension given in a recent result of Richter in 2018.
\end{abstract}


\date{\today}\maketitle
\section{Introduction and Notation}
Any set $A \subset \R^n$ fulfilling $A = t - A$ for some $t \in \R^n$ is called  \cemph{symmetric} and \cemph{0-symmetric} if $t=0$.
We denote the family of all \cemph{(convex) bodies} (full-dimensional compact convex sets) by $\K^n$ and the family of 0-symmetric bodies by $\K^n_0$.
For any $K \in \K^n$ the \cemph{gauge function} $\| \cdot \|_K : \R^n \to \R$ is defined as 
\[
\|x\|_K=\inf\{\rho>0:x\in\rho K\}. 
\]
In case $K \in \K^n_0$, we see that $\| \cdot \|_K$ defines a norm. 
However, even for a non-symmetric unit ball $K$, one may approximate the gauge function by the norms induced from symmetrizations of $K$ 
\begin{equation}\label{eq:norms}
\|x\|_{\conv(K\cup (-K))} \leq \|x\|_{K} \leq \|x\|_{K \cap (-K)}.
\end{equation}

It is natural to request that $K \cap (-K)= K= \conv(K\cup (-K))$ if $K$ is symmetric, which is true if and only if $0$ is the center of symmetry of $K$. This motivates the definition of a meaningful center for general bodies $K$. We introduce one of the most common asymmetry measures, which is best suited to our purposes, and choose the center matching it. 

The \cemph{Minkowski asymmetry} of $K$ is defined as 
\[
s(K):=\inf \{ \rho >0: K-c \subset \rho (c-K), \ c \in \R^n \},
\]
and a \cemph{Minkowski center} of $K$ is any $c \in \R^n$ such that $K-c \subset s(K)(c-K)$ \cite{BrG2}. Moreover, if $0$ is a Minkowski center, we say $K$ is \cemph{Minkowski centered}. It is well-known that $s(K) \in [1,n]$ for all $K \in \K^n$, with $s(K)=1$ if and only if $K$ is symmetric and $s(K)=n$ if and only if $K$ is a fulldimensional simplex \cite{Gr}.

For any $K, C\in\K^n$ we say $K$ is \cemph{optimally contained} in $C$, and denote it by $K\subset^{opt}C$, if $K\subset C$ and $K\not \subset_t \rho C$ for any $\rho \in [0,1)$. 

For $K \in \K^n$ we define $\alpha(K)>0$ such that $K  \cap (-K) \subset^{opt} \alpha(K) \, \conv(K \cup (-K))$.
Notice that there always exists some $x \in \R^n$ such that $ \alpha(K) \|x\|_{K \cap (-K)} = \|x\|_{\conv(K\cup (-K))}$, which means that we have equality for that $x$ in the complete chain in \eqref{eq:norms} if $\alpha(K)=1$. 

In \cite{BDG1} we started an investigation of the region of all possible values for the parameter $\alpha(K)$ for Minkowski centered $K \in\K^n$
in dependence of the asymmetry of $K$. It has been shown in \cite[Lemma 3.2]{BDG1} that for a Minkowski centered fulldimensional simplex $S$ we have
\[
\alpha(S)= 
\begin{cases}
1, & \text{if} \,\, n \, \text{is odd, and} \\
\frac{n}{n+1}, & \text{if} \,\, n \, \text{is even.}  
\end{cases}
\]
Moreover, it is shown in \cite[Theorem 1.7]{BDG1} that $\alpha(K) \geq \frac{2}{s(K)+1}$ for all Minkowski centered $K \in \K^n$, and that in the planar case $\alpha(K)=1$ implies $s(K) \leq \varphi$, where $\varphi = \frac{1+\sqrt{5}}{2} \approx 1.61$ denotes the \cemph{golden ratio}.

The main result of the present work is a complete 
description of the possible $\alpha$ values of $K$ in dependence of its asymmetry (c.f.~Figure \ref{fig:alpha-region}).

\begin{thm}\label{thm:PlanarCase}
Let $K \in\K^2$ be Minkowski centered. Then 
\[
\frac{2}{s(K)+1} \leq \alpha(K) \leq \min \left\{ 1, \frac{s(K)}{s(K)^2-1} \right\}.
\] 
Moreover, for every pair 
$(\alpha,s)$, such that $1 \leq s \leq 2$ and $\frac{2}{s+1} \leq \alpha \leq \min \left\{ 1, \frac{s}{s^2-1} \right\}$, there exists a Minkowski centered $K \in\K^2$, such that $s(K)=s$ and $\alpha(K)=\alpha$.
\end{thm}

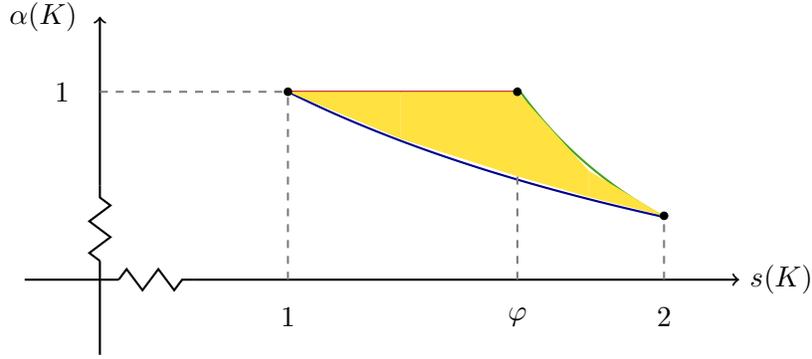
\begin{figure}[ht]
  \begin{tikzpicture}[scale=5]
    \draw[thick, discont] (0.05,0) -- (0.25,0);
    \draw[thick, discont] (0,0.05) -- (0,0.25);
    \draw [thick] (-0.2,0) -- (0.05,0);
    \draw [thick] (0,-0.2) -- (0,0.05);
    \draw[->] [thick] (0.25, 0) -- (1.7, 0) node[right] {$s(K)$};
    \draw[->] [thick] (0, 0.25) -- (0, 0.7);

    \draw [thick, dred,shift={(-0.5,-0.5)}] (1,1) -- (1.61,1);
    \draw[thick, dblue, domain=1:2, smooth, variable=\x, dblue,shift={(-0.5,-0.5)}]  plot ({\x}, {2/(\x+1)});
    \draw[thick, dgreen, domain=1.61:2, smooth, variable=\x,shift={(-0.5,-0.5)}]  plot ({\x}, {(\x)/((\x)^2-1)});
    \fill [fill=lgold, fill opacity=0.7, domain=1:1.3, variable=\x,shift={(-0.5,-0.5)}] (1,1) -- (1.3,1)-- (1.3,0.875)--(1,1);
    \fill [fill=lgold, fill opacity=0.7, domain=1.3:1.61, variable=\x,shift={(-0.5,-0.5)}] (1.3,1) -- (1.61,1)-- (1.61,0.775)--(1.3,0.875);
    \fill [fill=lgold, fill opacity=0.7, domain=1.3:1.61, variable=\x,shift={(-0.5,-0.5)}] (1.61,1) -- (1.8,0.79)--(1.8,0.72)-- (1.61,0.775)--(1.61,1);
    \fill [fill=lgold, fill opacity=0.7, domain=1.8:2, variable=\x,shift={(-0.5,-0.5)}] (1.8,0.79) -- (2,0.67)--(1.8,0.72)-- (1.8,0.79)--(1.8,0.79);
    
    \draw [thick, dashed,gray] (0,0.5) -- (0.5,0.5);
    \draw [thick, dashed,,gray] (1.11,0)--(1.11,0.28);
    \draw [thick, dashed,gray] (1.5,0) -- (1.5,0.17);
    \draw [thick, dashed,gray] (0.5,0) -- (0.5,0.5);
    
    \draw [fill,shift={(-0.5,-0.5)}] (1,1) circle [radius=0.01];
    \draw [fill,shift={(-0.5,-0.5)}] (1.61,1) circle [radius=0.01];
    \draw [fill,shift={(-0.5,-0.5)}] (2,0.67) circle [radius=0.01];
    
    \draw (-0.15,0.7) node {$\alpha(K)$};
    \draw (-0.1,0.5) node {$1$};
    \draw (0.5,-0.1) node {$1$};
    \draw (1.11,-0.1) node {$\varphi$};%
    \draw (1.5,-0.1) node {$2$};
     \end{tikzpicture}
     \caption{Region of possible values for the parameter $\alpha(K)$ for Minkowski centered $K \in\K^2$ (yellow): $\alpha(K) \geq \frac{2}{s+1}$ (blue); $\alpha(K) \leq 1$ for $s \leq \varphi$ (red), $\alpha(K) \leq \frac{s}{s^2-1}$ for $s \ge \varphi$ (green). Vertices are given by $0$-symmetric $K \in\K^2$ ($s=1$), the Golden House $\GH$ (see definition in Proposition \ref{prop:GoldenHouse}) ($s=\varphi$) and triangles ($s=2$).}
     \label{fig:alpha-region}
\end{figure}

Observe that from Theorem \ref{thm:PlanarCase} directly follows for Minkowski centered $K \in\K^2$ that
\begin{equation} \label{ineq:alpha_s}
\frac{1}{s(K)} \leq \frac{2}{s(K)+1} \leq \alpha(K),
\end{equation}
with equality if and only if $s(K)=1$.

While developing the proof of Theorem \ref{thm:PlanarCase}, we also made another interesting observation, for which we believe a separate theorem is justified.

\begin{thm}\label{thm:crossings}  
Let $K \in\K^2$ be Minkowski centered with $s(K) \geq \varphi$. Then the set
$\bd(K) \cap \bd(-K)$ consists of exactly 6 points.
\end{thm}

The value $\varphi$ above is the smallest possible quantity such that Theorem \ref{thm:crossings} stays true.
 
Consider $K \in \K^n$ and $C \in \K^n_0$. For $s \in \R^n \setminus \{0\}$ the \cemph{$s$-breadth} of $K$ w.r.t.~$C$ is the relative distance between the two parallel supporting hyperplanes of $K$ with normal vector $s$, i.e., 
\[
b_s(K,C) := \frac{ \max_{x,y \in K} s^T (x-y)}{\max_{x \in C} s^T x}.
\]
The minimal $s$-breadth 
\[w(K,C):= \min_{s \in \R^n \setminus \{0\}} b_s(K,C)\] 
and the maximal $s$-breadth 
\[D(K,C):=\max_{s \in \R^n \setminus \{0\}} b_s(K,C)\] 
are called \cemph{width} and \cemph{diameter} of $K$ w.r.t.~$C$, respectively. 
$K$ is said to be of \cemph{constant width} with respect to $C$, if $b_s(K,C)$ is constant in dependence of $s \in \R^n \setminus \{0\}$, i.e.,  
$w(K,C)=D(K,C)$ and  
$K-K = D(K,C)C$ \cite{Egg}.
Finally, $K$ is called \cemph{(diametrically) complete} w.r.t.~$C$, if any proper superset of $K$ has a greater diameter than $K$. 

The most famous example of a set of constant width w.r.t.~the euclidean ball is the Reuleaux triangle (see, e.g.~\cite{BrG}). One should recognize that constant width always implies completeness, but not the other way around \cite{Egg}. Minkowski spaces, in which all complete sets are of constant width are called \cemph{perfect}. Characterizing such spaces is still a major open question in convex geometry \cite{Egg,MoSch}.

The \cemph{circumradius} and the \cemph{inradius} of $K \in \K^n$ w.r.t.~$C \in \K^n$ are defined as
\begin{equation*}
    R(K,C):=\inf\{\rho>0: K \subset_t \rho C\} \quad \text{and} \quad r(K,C) := \sup\{\rho > 0 : \rho C \subset_t  K\},
\end{equation*}
where we write $K \subset_t C$, if there exists $t \in \R^n$, such that $K \subset C + t $. 
Whenever $C$ is symmetric, we have $D(K,C) = 2\max_{x,y \in K} R(\{ x,y \}, C) = \max_{x,y \in K} \|x-y\|_C$
and 
\[
\frac{w(K,C)}{2} \leq \frac{s(K)+1}{2} r(K,C) \le \frac{r(K,C) + R(K,C)}{2}  \le \frac{s(K)+1}{2s(K)} R(K,C) \leq \frac{D(K,C)}{2}
\]
(see \cite{BrKo}).

While $w(K,C)=D(K,C)$ characterizes constant width of $K$ w.r.t.~$C$, we know from \cite{BrG2} that all sets $K$, which are complete w.r.t.~$C\in \mathcal K^n_0$ fulfill the following chain of equalities
\[
\frac{s(K)+1}{2} r(K,C) =\frac{r(K,C) + R(K,C)}{2}  = \frac{s(K)+1}{2s(K)} R(K,C) = \frac{D(K,C)}{2}.
\]
However, this property does not characterize completeness, not even in the planar case. For instance, the sliced Reuleaux triangle and the hood, as defined in \cite{BrG} are pseudo-complete w.r.t.~the euclidean ball, but not of constant width. 

We say that $K$ is \cemph{pseudo-complete} w.r.t.~$C \in \mathcal K^n_0$ if $r(K,C)+R(K,C)=D(K,C)$, and denote by $\K^n_{ps,C}$ or $\K^n_{comp,C}$ the families of all $K \in \K^n$, which are pseudo-complete or complete w.r.t.~$C \in \mathcal K^n_0$, respectively.

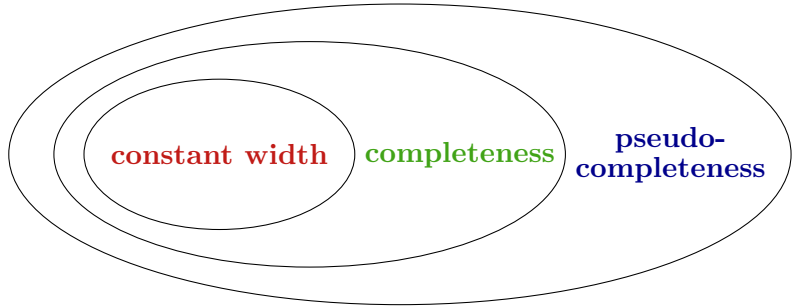
\vspace{0.3cm}
\begin{figure}[ht]
\begin{tikzpicture}[scale=2.0]
\draw \boundellipse{4.4,2}{-2.6}{1};
\draw \boundellipse{3.8,2}{-1.7}{0.75};
\draw \boundellipse{3.2,2}{-0.9}{0.5};
\draw[thick, dred]  (3.2,2) node {\textbf{constant width}};
\draw[thick, dgreen]  (4.8,2) node {\textbf{completeness}};
\draw[thick, dblue]  (6.2,2.1) node {\textbf{pseudo-}};
\draw[thick, dblue]  (6.2,1.9) node {\textbf{completeness}};
\end{tikzpicture}
\end{figure}

Recently, it has been shown in \cite{Ri} that the diameter-width-ratio for complete sets $K \in \K^n$ w.r.t.~$C\in\mathcal K^n_0$ is bounded from above by $\frac{n+1}{2}$. We sharpen this result for pseudo-complete (and therefore, for complete) sets taking the asymmetry $s(K)$ of $K$ into account.
\begin{thm}\label{thm:results_comp}
Let $K \in \K^n_{ps,C}$. Then 
\[
\frac{D(K,C)}{w(K,C)} \leq \frac{s(K)+1}{2}.
\]
Moreover, for $n>2$ odd and any $s \in [1,n]$ 
or for $n>2$ even and any $s \in [1,n-1]$ there exists $K \in \K^n_{comp,C}$ with $s(K)=s$, such that $\frac{D(K,C)}{w(K,C)} = \frac{s+1}{2}$. 
\end{thm}

While the above theorem sharpens the result in \cite{Ri}, it does not improve it in its absolute  bound of $\frac{n+1}{2}$. Doing this is indeed not possible in odd dimensions. We use a stability result near the simplex 
to show that for even $n$ the absolute bound can be improved to $\frac{n+1}{2} -\frac{1}{2n^4} + \mathcal{O}\left(\frac{1}{n^5}\right)$. 

\begin{thm}\label{thm:Improved_richter}
Let $K\in\mathcal K_{ps,C}^n$ and $n$ even. Then
\[
\frac{D(K,C)}{w(K,C)} \leq \frac{s_0+1}{2},
\]
where 
\[
s_0=\frac{n^4+n^3+2n^2+\sqrt{n^8+6n^7+17n^6+28n^5+28n^4+12n^3-4n^2-12n-4}}{2(n^3+2n^2+3n+1)}.
\]

\end{thm}

Euclidean spaces of any dimension as well as general planar Minkowski spaces are perfect. Thus, obviously, in all those spaces the diameter-width-ratio of complete sets is equal to one. However, as an application of Theorem \ref{thm:PlanarCase}, we are able to state better bounds than the ones given in Theorem \ref{thm:results_comp} and \ref{thm:Improved_richter} for pseudo-complete sets in the planar case. 

\begin{thm}\label{thm:results_pscomp}
Let $K \in \K^2_{ps,C}$. Then 
\[
\frac{D(K,C)}{w(K,C)} \leq \min \left\{\frac{s(K)+1}2, \frac{s(K)^2}{s(K)^2-1} \right\}.
\]
\end{thm}

We will also show that Theorem \ref{thm:results_pscomp} improves the absolute upper bound for the diameter-width-ratio of pseudo-complete planar sets from $\frac32$ (derived in Theorem \ref{thm:results_comp}) down to 
$\frac{1}{6} \left(4 + \sqrt[3]{19-3 \sqrt{33} } +\sqrt[3]{19+3 \sqrt{33} } \right)\approx 1.42$. 



Finally, we use the results on the 3-dimensional Blaschke-Santal\'o diagram w.r.t.~the circumradius, inradius, diameter, and width for convex bodies in the euclidean plane \cite{BrG}, to 
derive the optimal absolute upper bound for the diameter-width-ratio of pseudo-complete sets in that case.  
\begin{thm}\label{lem:dw_eucl}
Let $K \in \K^2_{ps,\B_2}$. Then 
\[
\frac{D(K,\B_2)}{w(K,\B_2)} \leq \frac12\left(1+\frac1r\right)\approx 1.135,
\]
where 
\[
r=\frac{\sqrt{t}}{2}-1+\sqrt{\frac{16}{\sqrt{t}}-t},\quad \text{and} \quad
t=\sqrt[3]{\frac{32}{9}} \left( \sqrt[3]{9+\sqrt{69}}+\sqrt[3]{9-\sqrt{69}} \right).
\]
Moreover, equality holds if $K$ is a hood (see Section 6 for details). 
\end{thm}

\section{Definitions and Propositions}
For any $X,Y \subset\R^n$, $\rho \in \R$ let $X+Y =\{x+y:x\in X,y\in Y\}$ be the \cemph{Minkowski sum} of $X$, $Y$ and  $\rho X= \{ \rho x: x \in X\}$ the \cemph{$\rho$-dilatation} of $X$, and abbreviate $(-1)X$ by $-X$. 
For any $X\subset\R^n$ let $\conv(X), \pos(X)$, $\lin(X)$, and $\aff(X)$ denote the \cemph{convex hull}, the \cemph{positive hull}, the \cemph{linear hull}, and the \cemph{affine hull} of $X$, respectively. 
A \cemph{segment} is the convex hull of $\{x,y\} \subset \R^n$, which we abbreviate by $[x,y]$. With $u^1, \dots, u^n$ we denote the \cemph{standard unit vectors} of $\R^n$. For every $X \subset \R^n$ let $\bd(X)$ and $\inter(X)$ denote the \cemph{boundary} and \cemph{interior} of $X$, respectively. Let us denote the \cemph{Euclidean norm} of $x\in\R^n$ by $\|x\|$, the \cemph{Euclidean unit ball} by $\B_2=\{x \in\R^n : \|x\|\leq 1\}$. 
In case $u^1,\dots,u^{n+1}\in\R^n$ are affinely independent, we say that $\conv(\{u^1,\dots,u^{n+1}\})$ is an \cemph{$n$-simplex}.

We recall the characterization of the optimal containment under homothety in terms of the touching conditions \cite[Theorem 2.3]{BrKo}. 
\begin{proposition}\label{prop:Opt_Containment}
Let $K,C\in\mathcal K^n$ and $K\subset C$. Then the following are equivalent:
\begin{enumerate}[(i)]
\item $K\subset^{opt}C$.
\item There exist $k\in\{2,\dots,n+1\}$, $p^j\in K\cap \bd(C)$, $j=1,\dots,k$, and $a^j$ outer normals of supporting
halfspaces of $K$ and $C$ at $p^j$,
such that $0\in\conv(\{a^1,\dots,a^k\})$.
\end{enumerate}
\end{proposition}

In the planar case Proposition \ref{prop:Opt_Containment} 
implies the following corollary (c.f.~\cite[Prop. 2.5]{BDG}).  

\begin{cor}\label{cor:zero-inside}
Let $K \in \K^2$ with $s(K) > 1$. Then $K$ is Minkowski centered if and only if there exist $p^1,p^2,p^3\in\bd(K)\cap(-\frac{1}{s(K)} K)$ and $a^i$, $i=1,2,3$, outer normals of supporting halfspaces of $K$ in $p^i$,
such that $0 \in \inter(\conv(\{p^1,p^2,p^3\}))$ and $0\in\conv(\{a^1,a^2,a^3\})$.
\end{cor}

For $K \in \K^n$ Minkowski centered we call any $p \in \bd(K) \cap \bd\left(-\frac{1}{s(K)}K\right)$ an \cemph{asymmetry point} of $K$, and any triple of asymmetry points, with the properties as stated in Corollary \ref{cor:zero-inside}, to be \cemph{well-spread}.

For any $a\in\R^n \setminus \{0\}$
and $\rho\in\R$, let $H^{\le}_{a,\rho} = \{x\in\R^n: a^Tx \leq \rho\}$ denote a \cemph{halfspace} with its boundary  being the \cemph{hyperplane} $H_{a,\rho}^==\{x\in\R^n: a^Tx = \rho\}$. 
Analogously, we define $H_{a,\rho}^\ge, H_{a,\rho}^<,H_{a,\rho}^>$. 
We say that the halfspace $H^{\le}_{a,\rho}$ \cemph{supports} $K \in\K^n$ in $q \in K$, if $K \subset H^{\le}_{a,\rho}$ and $q \in H_{a,\rho}^=$. We denote the set of all \cemph{extreme points} of $K$ by $\ext(K)$. 

Note that $ \alpha(K)=R\left(K  \cap (-K), \conv(K \cup (-K))\right)$, and we also define
\begin{equation*}
 \tau(K):=R\left(K  \cap (-K), \frac{K-K}{2}\right).
\end{equation*}

The following proposition combines \cite[Corollary 2.3]{BDG1} and \cite[Theorem 1.3]{BDG}. 

\begin{proposition}\label{prop:old_results}
Let $K \in\K^n$ be Minkowski centered and let $L$ be a regular linear transformation. Then the following are true: 
\begin{enumerate}[(i)]
    \item $\alpha(K)=\alpha(L(K))$ and $\tau(K)=\tau(L(K))$. 
    \item $\alpha(K)=\tau(K)=1$ if and only if there exist $p, -p \in \bd(K)$ and parallel halfspaces $H^{\le}_{a,\rho}$ and $H^{\le}_{-a,\rho}$ supporting $K$ in $p$ and $-p$, respectively.
\end{enumerate}
\end{proposition}

We recall another result from \cite{BDG} based on Proposition \ref{prop:old_results} 
about the equality cases for the upper bound in the inequality $\alpha(K) \le 1$ in the planar case.

\begin{proposition}\label{prop:GoldenHouse}
Let $K\in\mathcal K^2$ be Minkowski centered such that $\alpha(K)=\tau(K)=1$. Then $s(K)\leq \varphi$. Moreover, if $s(K)=\varphi$, there exists a linear transformation $L$ such that $L(K)=\mathbb{GH}$, where
\[
\mathbb{GH}:=\conv\left(\left\{\begin{pmatrix} \pm 1 \\ 0\end{pmatrix}, \begin{pmatrix} \pm 1 \\ -1\end{pmatrix}, \begin{pmatrix} 0 \\ \varphi \end{pmatrix} \right\} \right) 
\]
is the \cemph{golden house}.
\end{proposition}

\section{Geometry of the boundaries of sets} 
In 
this section we give the proof of Theorem \ref{thm:crossings}, describing the number of the intersection points of $\bd(K)$ and $\bd(-K)$ for a Minkowski centered $K \in\K^2$ with $s(K) > \varphi$. 
In order to do so, we provide the necessary definitions and show a lemma, which describes the locations of the asymmetry points from Corollary \ref{cor:zero-inside}. 
After this, we focus on the geometry of the touching points of $K\cap (-K)$ and $\alpha(K) \, \conv(K \cup (-K))$, which 
is needed for the proof of the main theorem.

For $K \in \K^2$ we call $z^1, z^2 \in \bd(K) \cap \bd(-K)$ \cemph{consecutive}, if there exists no point $z \in \bd(K) \cap \bd(-K) \cap \inter (\pos \{z^1, z^2\})$.

Even so we assume $s(K)>\varphi$ in this section, the arguments  keep valid as long as there is just a finite amount of 
points in $\bd(K) \cap \bd(-K)$.
In case of an \emph{infinite} sequence of points $\{z^i\}_{i\in\mathbb N}\subset\bd(K)\cap\bd(-K)$, there exists at least a subsequence of it, converging to a common boundary point $z^0$, such that  $K$ and $-K$ are commonly supportable in $z^0$. However, the latter would anyway imply $\alpha(K)=1$ by Proposition \ref{prop:old_results}, and thus by Proposition \ref{prop:GoldenHouse} is not possible for $s(K)>\varphi$.

\begin{lemma}\label{lem:crossings}
Let $K \in\K^2$ be Minkowski centered with $s(K) > \varphi$, let $z^1, z^2 \in \bd(K) \cap \bd(-K)$ be consecutive, and $H^\le_{a^1,\rho_1}$, $H^\le_{a^2,\rho_2}$ be two halfspaces supporting $K$ in $z^1$ and $z^2$, respectively. Moreover, let $p^1,p^2,p^3$ be well-spread asymmetry points of $K$. Then the following 
are true: 
\begin{enumerate}[(i)]
    \item If $p^i \in \pos\{z^1,z^2\}$, then $s(K) p^i \in H^>_{a^1,\rho_1}\cap H^>_{a^2,\rho_2}$.
    \item There exists either an asymmetry point of $K$ or of $-K$ in $\pos\{z^1,z^2\}$.
\end{enumerate}
\end{lemma}

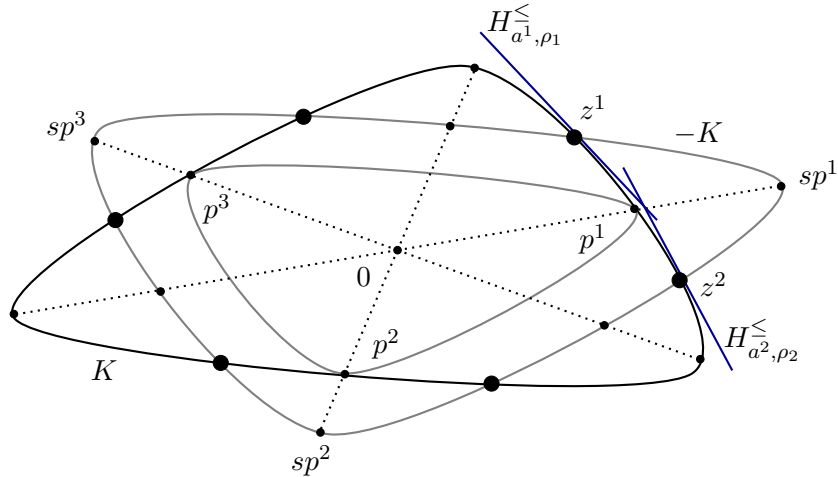
\begin{figure}[ht] 
\centering
\begin{tikzpicture}[scale=5]
\draw [gray, thick] plot [smooth cycle] coordinates {(-0.1,0.45) (0.5,-0.37) (1.7,0.3)};
\draw [ thick, rotate around={180:(0.68,0.12)}] plot [smooth cycle] coordinates {(-0.1,0.45) (0.5,-0.37) (1.7,0.3)};
\draw [gray, thick, scale=0.65, shift={(0.32,0.05)}] plot [smooth cycle] coordinates {(-0.1,0.45) (0.5,-0.37) (1.7,0.3)};

\draw [thick, dotted] (0.68,0.12) -- (1.69,0.29);
\draw [thick, dotted, rotate around={180:(0.68,0.12)}] (0.68,0.12) -- (1.69,0.29);
\draw [thick, dotted] (0.68,0.12) -- (0.47,-0.37);
\draw [thick, dotted, rotate around={180:(0.68,0.12)}] (0.68,0.12) -- (0.47,-0.37);
\draw [thick, dotted] (0.68,0.12) -- (-0.12,0.41);
\draw [thick, dotted, rotate around={180:(0.68,0.12)}] (0.68,0.12) -- (-0.12,0.41);

\draw [thick, dblue] (1.37,0.2) -- (0.9,0.7);
\draw [thick, dblue] (1.57,-0.2) -- (1.28,0.34);

\draw [fill] (0.68,0.12) circle [radius=0.01];
\draw [fill] (1.7,0.29) circle [radius=0.01];
\draw [fill, rotate around={180:(0.68,0.12)}] (1.7,0.29) circle [radius=0.01];
\draw [fill, rotate around={180:(0.68,0.12)}] (0.475,-0.365) circle [radius=0.01];
\draw [fill] (0.475,-0.365) circle [radius=0.01];
\draw [fill] (-0.125,0.41) circle [radius=0.01];
\draw [fill, rotate around={180:(0.68,0.12)}] (-0.125,0.41) circle [radius=0.01];
\draw [fill] (1.31,0.23) circle [radius=0.01];
\draw [fill, rotate around={180:(0.68,0.12)}] (1.31,0.23) circle [radius=0.01];
\draw [fill] (0.13,0.32) circle [radius=0.01];
\draw [fill, rotate around={180:(0.68,0.12)}] (0.13,0.32) circle [radius=0.01];
\draw [fill] (0.54,-0.21) circle [radius=0.01];
\draw [fill, rotate around={180:(0.68,0.12)}] (0.54,-0.21) circle [radius=0.01];
\draw [fill] (1.43,0.04) circle [radius=0.02];
\draw [fill,  rotate around={180:(0.68,0.12)}] (1.43,0.04) circle [radius=0.02];
\draw [fill] (0.43,0.475) circle [radius=0.02];
\draw [fill,  rotate around={180:(0.68,0.12)}] (0.43,0.475) circle [radius=0.02];
\draw [fill] (1.15,0.42) circle [radius=0.02];
\draw [fill,  rotate around={180:(0.68,0.12)}] (1.15,0.42) circle [radius=0.02];

\draw (0.59,0.05) node {$0$};
\draw (1.2,0.5) node {$z^1$};
\draw (1.52,0.03) node {$z^2$};
\draw (1.2,0.15) node {$p^1$};
\draw (1.8,0.32) node {$sp^1$};
\draw (0.65,-0.13) node {$p^2$};
\draw (0.45,-0.44) node {$sp^2$};
\draw (0.2,0.23) node {$p^3$};
\draw (-0.2,0.45) node {$sp^3$};
\draw (1.48,0.43) node {$-K$};
\draw (-0.1,-0.2) node {$K$};
\draw (1.02,0.72) node {$H^\le_{a^1,\rho_1}$};
\draw (1.65,-0.12) node {$H^\le_{a^2,\rho_2}$};
\end{tikzpicture}
\caption{Construction from Lemma \ref{lem:crossings} and Theorem \ref{thm:crossings}: 
$K$ (black), $-K$ and $-\frac{1}{s(K)}K$  (gray), 
$\bd(K) \cap \bd(-K)$ (big black dots), halfspaces $H^\le_{a^1,\rho_1}$, $H^\le_{a^2,\rho_2}$  supporting $K$ at consecutive points $z^1$ and $z^2$, respectively, and $p^1, p^2, p^3$ well-spread asymmetry points of $K$.
}
\end{figure}

\begin{proof}
We start proving (i) and define $s:=s(K)$. Notice that $(-K) \cap H^>_{a^1,\rho_1} \subset\pos(\{z^1,z^2\})$ and that, since $0\in\conv(\{p^1,p^2,p^3\})$, between any two of those three asymmetry points, we find points from $\bd(K)\cap\bd(-K)$. To check the latter, consider e.g.~$p^1$ and $p^2$. Since $sp^1,-p^3 \in\bd(-K)\cap\pos(\{p^1,p^2\})$ and $p^1,-sp^3\in\bd(K)\cap\pos(\{p^1,p^2\})$,
we find a point from $\bd(K)\cap\bd(-K)$ in $\inter(\pos(\{p^1,-p^3\})) \subset \inter(\pos\{p^1,p^2\})$. 


Suppose (i) were wrong, then we may 
assume w.l.o.g.~that $p^1 \in \pos\{z^1,z^2\}$ and 
$sp^1 \in H_{a^1,\rho_1}^\le$.
Because of the observation before it follows that $p^2,p^3 \notin \pm\pos(\{z^1,z^2\})$. 
Defining $K':=K \cap (-H_{a^1,\rho_1}^\le)$, we obtain $-sp^j\in K'$, $j=1,2,3$, and therefore $-\frac1s K' \subset K'$, with $p^1,p^2,p^3$ being well-spread asymmetry point of $K'$, too. 
It follows from Proposition \ref{prop:Opt_Containment} that $K'$ is still Minkowski centered and $s(K')=s$. Moreover, the halfspaces $ H^\le_{\pm a^1,\rho_1}$ support $K'$ in the points $\pm z^1$, respectively, implying $\alpha(K') = 1$ by Proposition \ref{prop:old_results}. However, this means $K'$ is Minkowski centered with $s(K') > \varphi$ and $\alpha(K') = 1$, contradicting Proposition \ref{prop:GoldenHouse}.

The proof of (ii) is completely analogous to the one above, starting here with the assumption that there is no asymmetry point between $z^1,z^2$. 
\end{proof}

\begin{proof}[Proof of Theorem \ref{thm:crossings}] 
Since $s:=s(K)>1$, 
there exists a triple of well-spread asymmetry points $p^1,p^2,p^3$ of $K$. Then $-p^i\in\pos(\{p^j,p^k\})$, whenever $\{i,j,k\}=\{1,2,3\}$. In particular, this means that we find at least one point from $\bd(K)\cap\bd(-K)$ within each cone $\inter(\pos(\{p^i,-p^j\}))$, $i \neq j$, which shows that $\bd(K)\cap\bd(-K)$ contains at least 
6 points.

Now, let us assume that $\bd(K) \cap \bd(-K)$ contains more than 6 points. Using the pigeonhole principle, we see that there must exist
two consecutive points $z^1,z^2\in\bd(K)\cap\bd(-K)$, such that $\pm p^i\notin \pos(\{z^1,z^2\})$, $i=1,2,3$. 
Let us assume w.l.o.g.~that there exists some
$z\in\bd(K)\cap\inter(-K)\cap\inter(\pos(\{z^1,z^2\}))$. Let $H^\le_{a,\rho}$ be a supporting halfspace of $K$ in $z$. 
Then, $\inter(-K)\cap H^>_{a,\rho} \subset\pos(\{z^1,z^2\})$, and therefore $K \setminus\{-\pos(\{z^1,z^2\})\} \subset K\cap(-H^\le_{a,\rho}) =:K'$.
It follows that $p^i,-sp^i\in K'$, $i=1,2,3$, and therefore that $s(K') = s \geq\varphi$ by the assumption of the theorem. Moreover, since $\pm H^\le_{a,\rho}$ are parallel supporting halfspaces of $K'$ in $\pm z$, we have $\alpha(K')=1$ by Proposition \ref{prop:old_results}.

For $s >\varphi$ this would directly contradict Proposition \ref{prop:GoldenHouse}. 
In case of $s= \varphi$ we have that $K'$ is Minkowski centered with $\alpha(K')=1$ and $s(K') =\varphi$. Thus, by the equality case of Proposition \ref{prop:GoldenHouse}, we obtain that $K'$ equals the golden house $\mathbb{GH}$ up to a linear transformation. However, this constradicts our assumption of more than 6 points in $\bd(K')\cap\bd(-K')$.
\end{proof}

In case of $1<s< \varphi$, 
the set $\bd(K) \cap \bd(-K)$ may consist of an uncountably infinite amount of points, an arbitrarily large finite amount, or a small number of points as well.

\begin{example}\label{rem:alpha1s1varphi}
\begin{enumerate}[(i)]
\item For any $s \in [1,\varphi)$ there exists a Minkowski centered $K \in\K^2$ with $s(K)=s$, such that the set $\bd(K) \cap \bd(-K)$ is uncountable. 

    Let $K_t:=\conv\left(\left\{\smallmat{\pm 1 \\ 0}, \smallmat{\pm 1 \\ -1}, \smallmat{0 \\ t} \right\} \right) \in \K^2$ with $t\in[0,\varphi)$. 
It is not hard to verify that $K_t-x_t\subset^{opt}s(-K_t+x_t)$, where
\[
s=\frac{t+\sqrt{9t^2+12t+4}}{2(t+1)}\quad\text{and}\quad x_t=\begin{pmatrix} 0 \\ \frac{t-s}{s+1}\end{pmatrix}.
\]
Thus, we obtain $\{s(K_t):t\in[0,\varphi)\}=[1,\varphi)$, and since $\pm u^1 \in \bd(K_t)$, $t\in[0,\varphi)$, with $H^=_{u^1,\pm 1}$ supporting $K_t$ in $\pm u^1$, we have $\alpha(K_t)=1$ for every $t\in[0,\varphi)$.

Finally, $\left[ \smallmat{ 1 \\ \frac{t-s}{s+1} }, \smallmat{1 \\ \frac{s-t}{s+1}}\right] \subset \bd(K_t)\cap\bd(-K_t)$, thus the boundaries of $\pm K_t$ possess \emph{infinitely} many common points.
\item Let $K \in \K^2$ be a Minkowski centered regular $k$-gon with $k \geq 5$ odd. By \cite[Example 4.2]{BDG1}, $s(K)=\frac{1}{\cos(\frac{\pi}{k})} < \varphi$ and $\bd(K)\cap\bd(-K)$ consists of $2k$ points.
\item Let $S$ be a (regular) simplex and $K= S \cap (-sS)$ with $s \in [1,\varphi)$. It is easy to see that $s(K)=s$, that $K \cap (-K)= S \cap (-S)$, and that $\bd(K) \cap \bd(-K)$ contains exactly 6 points. 
\end{enumerate}
\end{example}

We discuss the locations of the touching points of $K\cap (-K)$ and $\alpha(K) \, \conv(K \cup (-K))$. 

\begin{lemma}\label{lem:Charact_Opt_Means}
Let $K\in\K^2$ be Minkowski centered with $s(K)>1$ and $p \in \bd(K\cap (-K) )  \cap \bd(\alpha(K) \, \conv(K \cup (-K)))$. Then one of the following must be true:
\begin{enumerate}[(i)]
  \item $p \in \bd(K) \cap \bd(-K)$ or
  \item $\frac{1}{\alpha(K)} p \not \in \conv (K \cup (-K)) \setminus (K \cup (-K))$.
\end{enumerate}
Moreover, if (i) is fulfilled, there exist halfspaces $H^\le_{a, \rho}$ and $H^\le_{-a, \rho}$, such that 
each of them supports both, $K$ and $-K$, while $H^\le_{a, \alpha(K) \rho}$ and $H^\le_{-a, \alpha(K) \rho}$ support $K \cap (-K)$ in $p$ and $-p$, respectively.
\end{lemma}

Note that Case (ii) of Lemma \ref{lem:Charact_Opt_Means} is equivalent to one of the following two statements getting true:
\[p \in \bd(K) \text{ with } \frac{1}{\alpha(K)} p \in \bd(-K) \quad \text{or} \quad
p \in \bd(-K) \text{ with } \frac{1}{\alpha(K)} p \in \bd(K).\]

\begin{figure}[ht] 
\centering
\begin{tikzpicture}[scale=5]
\draw [gray, thick] plot [smooth cycle] coordinates {(-0.1,0.45) (0.5,-0.37) (1.7,0.3)};
\draw [ thick, rotate around={180:(0.68,0.12)}] plot [smooth cycle] coordinates {(-0.1,0.45) (0.5,-0.37) (1.7,0.3)};
\draw [thick, dblue] (1.77,0.4) -- (1.41,-0.35);
\draw [thick, dblue, shift={(-0.165,0)}] (1.77,0.4) -- (1.41,-0.35);
\draw [thick, dblue, shift={(-1.67,0.15)}] (1.77,0.4) -- (1.41,-0.35);
\draw [thick, dblue, shift={(-1.84,0.15)}] (1.77,0.4) -- (1.41,-0.35);
\draw [thick, dotted] (0.68,0.12) -- (1.59,0.03);
\draw [thick, dotted, rotate around={180:(0.68,0.12)}] (0.68,0.12) -- (1.59,0.03);

\draw [fill] (0.68,0.12) circle [radius=0.01];
\draw [fill] (1.43,0.04) circle [radius=0.02];
\draw [fill, rotate around={180:(0.68,0.12)}] (1.43,0.04) circle [radius=0.02];

\draw (0.59,0.05) node {$0$};
\draw (1.35,0.09) node {$p$};
\draw (0.05,0.13) node {$-p$};
\draw (1.2,-0.37) node {$H^\le_{a, \alpha(K) \rho}$};
\draw (0.23,0.58) node {$H^\le_{-a, \alpha(K) \rho}$};
\draw (1.55,-0.3) node {$H^\le_{a, \rho}$};
\draw (-0.2,0.57) node {$H^\le_{-a, \rho}$};
\draw (1.48,0.43) node {$-K$};
\draw (-0.1,-0.2) node {$K$};

\end{tikzpicture}
\caption{Construction for Case (i) of Lemma \ref{lem:Charact_Opt_Means}: $K$ (black), $-K$ (gray), $H^\le_{a, \rho}$, $H^\le_{-a, \rho}$ and $H^\le_{a, \alpha(K) \rho}$, $H^\le_{-a, \alpha(K) \rho}$ (blue) .   
}
\label{fig:three}
\end{figure}
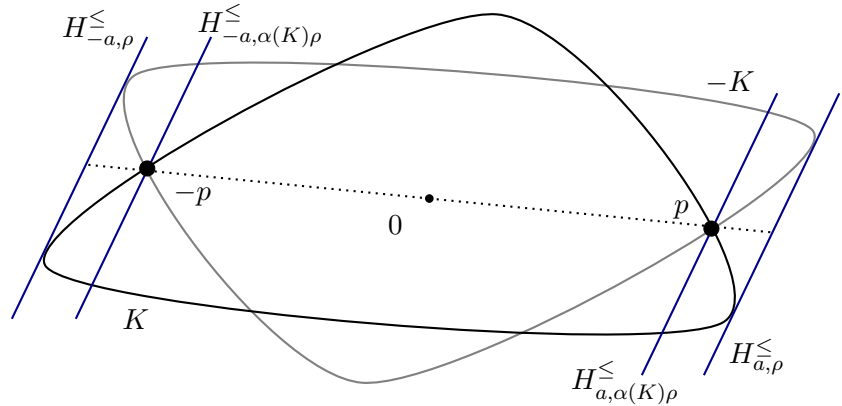

\begin{proof}[Proof of Lemma \ref{lem:Charact_Opt_Means}]
Let us start observing that if $\alpha := \alpha(K) = 1$, we have $p\in\bd(\conv(K\cup(-K))$. Hence, $p\in\bd(K)\cap\bd(-K)$, i.e.~(i) applies. 
Hence, we may assume $\alpha <1$ for the rest of the proof.

We first show the \enquote{moreover}-part of the statement. Thus, assume $p \in \bd(K) \cap \bd(-K)$ and let $s:=s(K)>1$. 
Since $\alpha <1$, we have $\pm \frac{1}{\alpha} p \not \in K$, but there exist $x,-y \in \ext(K)$ such that $\frac{1}{\alpha} p \in [x,y] \subset \bd(\conv(K \cup (-K)))$ and a halfspace $H^\le_{a, \rho}$ supporting $\conv(K \cup (-K))$ in $\frac{1}{\alpha} p$ and thus in $[x,y]$. 

Since $K \cap (-K)$ and $\conv(K \cup (-K))$ are both symmetric, the halfspace $H^\le_{-a, \rho}$ supports $\conv(K \cup (-K))$ in $-\frac{1}{\alpha} p$. Using $\frac{1}{\alpha} (K \cap (-K)) \subset \conv(K \cup (-K)) \subset H^\le_{a, \rho}$, we see that $H^\le_{\pm a, \alpha\rho}$ support $K \cap (-K)$ in $\pm p$. 

Now, for the sake of a contradiction, let us assume (i) and (ii) are wrong. Doing so, we may assume w.l.o.g.~that $p \in \bd(-K) \cap \inter(K)$ and $\frac{1}{\alpha} p \in \conv (K \cup (-K)) \setminus K$.
Then there exist $x,-y \in \bd(-K) \setminus K$,
such that $\frac1\alpha p \in [x,y] \subset \conv(K \cup (-K)).$
Let $z \in \bd(K) \cap \bd(-K) \cap \pos(\{x,y\})$ and
$\xi \in \R$ be such that  $\xi z \in [x,y]$. 
Since $\frac1\alpha p \in [x,y] \setminus K$, we have $\xi > 1$.
Obviously, $z\in\bd(K)\cap\bd(-K)$ and $x\in \bd(-K)\setminus K$ together imply $(\bd(-K) \cap \inter(\pos(\{x,z\}))) \cap K = \emptyset$.

Since $p\in\bd(-K)\cap\inter(K)$, we obtain $p,x,z \in \bd(-K)$ with $z \in \pos(\{p,x\})$, and, due to the convexity of $-K$, $z\in x+\pos(\{p-x,\frac1\alpha p-x\})$.
It follows that $\frac{1}{\alpha \xi} p \not \in -K$ and, since $p \in \bd(-K)$, we have $\frac1\alpha>\xi$. Together with $\frac1\alpha(K\cap(-K))\subset\conv(K\cup(-K))$, $z\in K\cap(-K)$, and $\xi z\in\bd(\conv(K\cup(-K))$, this implies the desired contradiction.
\end{proof}

We present a family of sets, where the situation described in Lemma \ref{lem:Charact_Opt_Means} (ii) happens.

\begin{example}\label{ex:case_two}
Let $S=\conv(\{p^1,p^2,p^3\})$ be a regular Minkowski centered triangle with $R(S, \B_2)=1$ and $M:=\left\{ \smallmat{ \rho_1 \\ \rho_2} : 1 \leq \rho_1 \leq 2, \frac{\rho_1^2-\rho_1+1}{\rho_1+1} \leq \rho_2 \leq \frac{\rho_1}{2} \right\}$. We define 
\[
K := K_{\rho_1,\rho_2} := \conv(((-S) \cap \rho_1 S) \cup \rho_2 S) \text{ with } \smallmat{ \rho_1 \\ \rho_2 }  \in M. 
\]


    
    
    

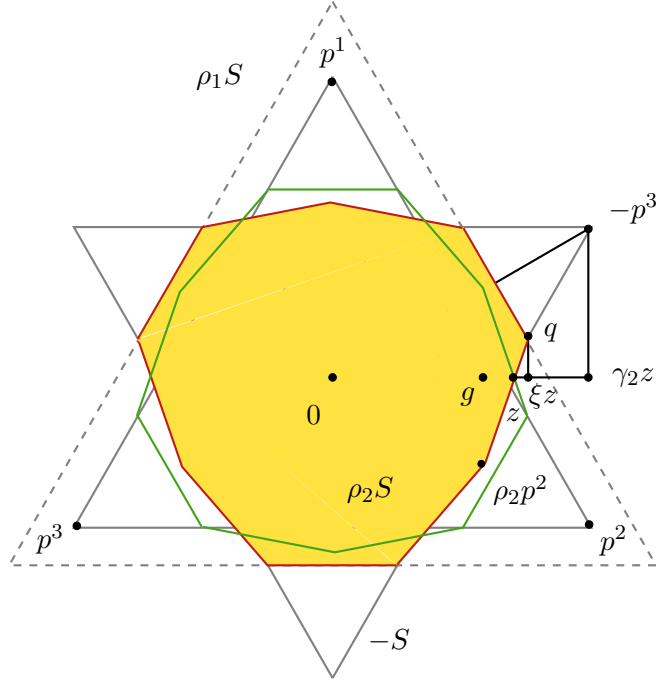
\begin{figure}[ht]
\begin{tikzpicture}[scale=5]
    \draw [thick, scale=0.8, gray, rotate around={180:(0,0)}] (0,1) -- (0.86,-0.5)--(-0.86,-0.5)--(0,1);
    \draw [thick, scale=0.8, gray] (0,1) -- (0.86,-0.5)--(-0.86,-0.5)--(0,1);
    \draw [thick, gray, dashed] (0,1) -- (0.86,-0.5)--(-0.86,-0.5)--(0,1);
    \draw [thick, scale=0.46, gray, dashed] (0,1) -- (0.86,-0.5)--(-0.86,-0.5)--(0,1);
    \draw [thick] (0,0) -- (0.68,0.395)-- (0.68,0);
    \draw [thick] (0,0) -- (0.68,0);
    \draw [thick] (0,0) -- (0.52,0.11)--(0.52,0);
    \draw [thick] (0,0) -- (0.4,-0.23)--(0.4,0);
    
    \draw [thick,dred, rotate around={180:(0,0)}, fill=lgold, fill opacity=0.7] (-0.17,0.5) -- (0.173,0.5)--(0.4,0.238)--(0.517,-0.1);
    \fill [fill=lgold, fill opacity=0.7] (0.35,0.4)-- (0.172,-0.5)--(-0.52,0.1);
    \draw [thick, rotate=60,dred, fill=lgold, fill opacity=0.7] (-0.17,0.5) -- (0.173,0.5)--(0.4,0.238)--(0.517,-0.1);
    \draw [thick,dred, rotate=-60, fill=lgold, fill opacity=0.7] (-0.17,0.5) -- (0.173,0.5)--(0.4,0.238)--(0.517,-0.1);
   \draw [thick,dgreen, rotate around={120:(0,0)}] (-0.17,0.5) -- (0.173,0.5)--(0.4,0.238)--(0.517,-0.1);
   \draw [thick, rotate=0,dgreen] (-0.17,0.5) -- (0.173,0.5)--(0.4,0.238)--(0.517,-0.1);
   \draw [thick,dgreen, rotate around={240:(0,0)}] (-0.17,0.5) -- (0.173,0.5)--(0.4,0.238)--(0.517,-0.1);
   \draw [fill] (0,0) circle [radius=0.01];
   \draw [fill] (0.52,0.11) circle [radius=0.01];
   \draw [fill] (0.52,0) circle [radius=0.01];
   \draw [fill] (0.68,0.395) circle [radius=0.01];
   \draw [fill] (0.68,0) circle [radius=0.01];
   \draw [fill] (0.48,0) circle [radius=0.01];
   \draw [fill, scale=0.46] (0.86,-0.5) circle [radius=0.02];
   \draw [fill] (0.4,0) circle [radius=0.01];
   \draw [fill] (-0.68,-0.395) circle [radius=0.01];
   \draw [fill, rotate around={120:(0,0)}] (-0.68,-0.395) circle [radius=0.01];
   \draw [fill, rotate around={-120:(0,0)}] (-0.68,-0.395) circle [radius=0.01];

   \draw (-0.3,0.8) node {$\rho_1 S$};
   \draw (0.15,-0.7) node {$-S$};
   \draw (-0.05,-0.1) node {$0$};
   \draw (0.1,-0.3) node {$\rho_2 S$};
   \draw (0.8,0.45) node {$-p^3$};
   \draw (-0.75,-0.43) node {$p^3$};
   \draw[rotate around={-120:(0,0)}] (-0.75,-0.43) node {$p^1$};
   \draw[rotate around={120:(0,0)}] (-0.75,-0.43) node {$p^2$};
   \draw (0.8,0) node {$\gamma_2 z$};
   \draw (0.56,-0.05) node {$\xi z$};
   \draw (0.36,-0.05) node {$g$};
   \draw (0.58,0.12) node {$q$};
   \draw (0.5,-0.3) node {$\rho_2 p^2$};
   \draw (0.48,-0.1) node {$z$};
   \end{tikzpicture}
   \caption{Constructions used in Example \ref{ex:case_two}: $K$ (yellow), $-K$ (green), $S$ and $-S$ (gray), $\rho_2 S$ and $\rho_1 S$ (gray dashed).
   }
\end{figure}

Let further $z \in \bd(K) \cap \bd(-K) \cap  \pos(\{p^2,-p^3\})$ with $\xi z \in \bd (\conv( K \cup (-K)))$ for some $\xi >1$. Let $\gamma_1 \leq 1$ and $\gamma_2 \geq 1$ be such that $\gamma_1 z \in \bd(S) \cap \bd(-S)$ and $\gamma_2 z \in \bd(\conv(S \cup (-S)))$, respectively. Then, since $R(S, \B_2)=1$, we have $\|\gamma_1 z\|=\frac{\sqrt{3}}{3}$ and $\|\gamma_2 z\|=\frac{\sqrt{3}}{2}$. Choose $q$ to be the vertex of $K$, such that $q \in \inter(\pos(\{p^2,-p^3\}))$.

 Now, since $\frac{\|-p^3-q\|}{\|-p^3-\gamma_1 z\|}=\frac{\|-p^3-(-\frac{\rho_1}{2} p^3)\|}{\|-p^3-(-\frac{1}{2} p^3)\|}$, we have 
\begin{equation*}
\frac{\|-p^3-q\|}{\|-p^3-\gamma_1 z\|}=\frac{1-\frac{\rho_1}{2}}{\frac{1}{2}}=2-\rho_1.
\end{equation*}
Therefore,
\begin{equation*}
\frac{\|\gamma_2 z-\gamma_1 z\|}{\|\xi z-\gamma_1 z\|}=1+\frac{\|-p^3-q\|}{\|q-\gamma_1 z\|}=1+\frac{2-\rho_1}{\rho_1-1}= \frac{1}{\rho_1-1}. 
\end{equation*}
We have
\begin{equation*}
\|\xi z\|= \|\gamma_1 z\|+ \|\xi z-\gamma_1 z\|=
\frac{\sqrt{3}}{3} +\frac{\sqrt{3}}{6} (\rho_1-1)= \frac{\sqrt{3}}{6} (\rho_1+1).
\end{equation*}
Let $g$ be the orthogonal projection of $\rho_2 p^2$ on $\lin(\{z\})$. 

Since $\| -\frac{\rho_1}{2} p^3 - q\|= \tan(\frac{\pi}{6}) \cdot \| -p^3-\left(-\frac{\rho_1}{2} p^3\right)\|  =\frac{\sqrt{3}}{3} \left( 1- \frac{\rho_1}{2} \right)$, 
we obtain from the Pythagorean theorem  
\begin{equation*}
\|q\|= \sqrt{\left(\frac{\rho_1}{2} \right)^2 + \left(\frac{\sqrt{3}}{3} \left(1- \frac{\rho_1}{2} \right) \right)^2 }= \sqrt{ \frac{\rho_1^2-\rho_1+1}{3}}.
\end{equation*}
Now, we calculate 
\begin{equation*}
\|q-\xi z\|= \sqrt{\|q\|^2-\|\xi z\|^2}= \sqrt{\frac{\rho_1^2-\rho_1+1}{3}- \left(\frac{\sqrt{3}}{6} (\rho_1+1) \right)^2}= \frac{\rho_1-1}{2}.
\end{equation*}
Since $\|g\|=\frac{\sqrt{3}}{2} \rho_2$, we have 
\begin{equation*}
\|g-\rho_2 p^2\|= \sqrt{\|\rho_2 p^2\|^2-\|g\|^2}= \sqrt{\rho_2^2-\frac{3}{4}\rho_2^2}= \frac{\rho_2}{2}.
\end{equation*}
Since $q,z,\rho_2p^2$ are col-linear, we have $\frac{\|g- z\|}{\|z-\xi z\|}= \frac{\|g- \rho_2 p^2\|}{\|q-\xi z\|}$, and thus
\begin{equation*}
\frac{\|g- z\|}{\|z-\xi z\|}=\frac{\frac{\rho_2}{2}}{\frac{\rho_1-1}{2}}=\frac{\rho_2}{\rho_1-1} . 
\end{equation*}

Using this fact, we obtain
\begin{equation*}
\|\xi z\|= \|g\|+\|g-z\|+\|\xi z-z\|= \frac{\sqrt{3}}{2} \rho_2+ \left( \frac{\rho_2}{\rho_1-1}+1 \right) \|\xi z-z\|.
\end{equation*}
Therefore, remembering that $\|g\|=\frac{\sqrt{3}}{2} \rho_2$, 
\begin{equation*}
\|\xi z-z\|=\frac{\|\xi z\|-\|g\|}{\frac{\rho_2}{\rho_1-1}+1} = \frac{\|\xi z\|-\frac{\sqrt{3}}{2} \rho_2}{\frac{\rho_2}{\rho_1-1}+1}
\end{equation*}
and 
\begin{align*}
 \frac{\|z\|}{\|\xi z\|}&=\frac{\|\xi z\|-\|\xi z-z\|}{\|\xi z\|}= 1-\frac{1}{\frac{\rho_2}{\rho_1-1}+1}+\frac{ \frac{\sqrt{3}}{2} \rho_2}{\left( \frac{\rho_2}{\rho_1-1}+1 \right) \frac{\sqrt{3}}{6} (\rho_1+1)}\\
 &=\frac{\rho_2 (\rho_1-1)}{\rho_1+\rho_2-1} \left(\frac{1}{\rho_1-1}+\frac{3}{\rho_1+1} \right)= \frac{2\rho_2 (2\rho_1-1)}{(\rho_1+\rho_2-1)(\rho_1+1)}.
\end{align*}
Since $\rho_2 p^2 \in \bd(K)$ is a vertex of $\rho_2 S$
and $\frac{\rho_1}{2} p^2 \in \bd(\rho_1 S) \cap \bd(-K)$, we have 
$\frac {1}{\alpha(K)} \leq \frac{\rho_1}{2 \rho_2}$.

Thus, notice that 
\begin{equation*}
\frac1{\alpha(K)} = \min \left\{ \xi, \frac{\rho_1}{2 \rho_2}\right\}= \min \left\{  \frac{(\rho_1+\rho_2-1)(\rho_1+1)}{2\rho_2 (2\rho_1-1)}, \frac{\rho_1}{2 \rho_2} \right\}.
\end{equation*}
Since
\begin{equation*}
\frac{\rho_1}{2 \rho_2} \leq \frac{(\rho_1+\rho_2-1)(\rho_1+1)}{2\rho_2 (2\rho_1-1)}
\end{equation*}
is equivalent to 
\begin{equation}\label{eq:ex_4.2}
\frac{\rho_1^2-\rho_1+1}{\rho_1+1} \leq \rho_2,
\end{equation}
we obtain 
$\alpha(K)=\frac{\rho_1}{2 \rho_2}$,  
whenever \eqref{eq:ex_4.2} is fulfilled. 
Since $\frac{\rho_1^2-\rho_1+1}{\rho_1+1} \leq \frac{\rho_1}{2}$, for any fixed $\rho_1\in[1,2]$ we can select $\rho_2$, such that $\frac{\rho_1^2-\rho_1+1}{\rho_1+1} \leq \rho_2 \leq \frac{\rho_1}{2}$, and thus selecting any $\smallmat{\rho_1 \\ \rho_2 }  \in M$ would give us examples of bodies fulfilling the condition above.
\end{example}

\section{The Proof of Theorem \ref{thm:PlanarCase}}

\begin{lemma}\label{lem:beforePlanarCase}
Let $K\in\mathcal K^2$ be Minkowski centered with $s(K)>\varphi$ and $p^1,p^2,p^3$ a triple of well-spread asymmetry points of $K$. 
\begin{enumerate}[a)]
\item Then there exists 
a Minkowski centered $K'\in\mathcal K^2$ with
\begin{enumerate}[(i)]
  \item $\conv(K' \cup (-K')))=\conv( \{ \pm sp^1, \pm sp^2, \pm sp^3\} )$,
  \item $\alpha(K') \geq \alpha(K)$, 
  \item $s(K')=s(K)$, and
  \item $p \in \bd(K' \cap (-K')) \cap \alpha(K') (\bd (\conv(K' \cup (-K'))))$ implies $p \in \bd(K') \cap \bd(-K')$.
\end{enumerate}
\item Taking $p$
as above and defining $d^1$ to be the intersection point of $\aff (\{-p,p^3\})$ and $\aff (\{p,p^2\})$, there exists $\gamma \le 1$, such that the set
\begin{equation*}
 \bar K:=\conv(\{\pm p,p^2,p^3,\gamma s(K') d^1,-s(K')p^2, -s(K')p^3\})   
\end{equation*}
fulfills
\begin{enumerate}[(i)]
  \item $\alpha(\bar K) \geq \alpha(K')$, 
  \item $s(\bar K)=s(K')$, and
  \item $p \in \bd(\bar K \cap (-\bar K)) \cap \alpha(K') (\bd (\conv(\bar K \cup (-\bar K))))$.
\end{enumerate}
\end{enumerate}

\end{lemma}
\begin{proof} 
\begin{enumerate}[a)]
\item 
Let $s:=s(K)$. Since $s>\varphi$, 
there exists a triple of well-spread asymmetry points $p^1,p^2,p^3$ of $K$ by Corollary \ref{cor:zero-inside} and by Theorem \ref{thm:crossings}, 
we know that $\bd(K) \cap \bd(-K)$ consists of exactly 6 points. Moreover, by Lemma \ref{lem:crossings} there exist consecutive points $z^{i,1}, z^{i,2} \in \bd(K) \cap \bd(-K)$  with $-p^i \in \pos \{z^{i,1}, z^{i,2}\}$, $i=1,2,3$. Let the halfspaces $H^\le_{a^{i,1},\rho_{i,1}}$, $H^\le_{a^{i,2},\rho_{i,2}}$ be defined such that $z^{i,1}, -sp^i \in H^=_{a^{i,1},\rho_{i,1}}$, $z^{i,2}, -sp^i \in H^=_{a^{i,2},\rho_{i,2}}$, $i=1,\dots,3$, and $0 \in H^\le_{a^{i,1},\rho_{i,1}}\cap H^\le_{a^{i,2},\rho_{i,2}}$. 
We define 
\begin{equation}\label{eq:K'}
K':=K\cap\bigcap_{i=1,\dots,3}\left(H^\le_{a^{i,1},\rho_{i,1}}\cap H^\le_{a^{i,2},\rho_{i,2}}\right).
\end{equation}

\begin{figure}[ht] 
\centering
\begin{tikzpicture}[scale=5]
\draw [thick] plot [smooth cycle] coordinates {(-0.1,0.45) (0.5,-0.37) (1.7,0.3)};
\draw [thick, rotate around={180:(0.68,0.12)}] plot [smooth cycle] coordinates {(-0.1,0.45) (0.5,-0.37) (1.7,0.3)};
\draw [red, thick, scale=0.65, shift={(0.42,0.08)}, rotate around={180:(0.68,0.12)}] plot [smooth cycle] coordinates {(-0.1,0.45) (0.5,-0.37) (1.7,0.3)};

\draw [thick, dotted] (0.68,0.12) -- (1.69,0.29);
\draw [thick, dotted, rotate around={180:(0.68,0.12)}] (0.68,0.12) -- (1.69,0.29);
\draw [thick, dotted] (0.68,0.12) -- (0.47,-0.37);
\draw [thick, dotted, rotate around={180:(0.68,0.12)}] (0.68,0.12) -- (0.47,-0.37);
\draw [thick, dotted] (0.68,0.12) -- (-0.12,0.41);
\draw [thick, dotted, rotate around={180:(0.68,0.12)}] (0.68,0.12) -- (-0.12,0.41);


\fill [fill=lgold, fill opacity=0.7, domain=1:1.3, variable=\x, rotate around={180:(0.68,0.12)}] (0.42,0.47)--(0.89,0.6) -- (1.17,0.41)--(1.3,0.24)--(1.43,0.04)--(1.485,-0.17)--(0.93,-0.232)--(0.4,-0.2)--(0.2,-0.17)--(-0.35,-0.05)--(-0.07,0.2)--(0,0.25);

\draw [thick, dred] (1.15,0.42) -- (1.7,0.29) -- (1.43,0.04);
\draw [thick, dred, rotate around={180:(0.68,0.12)}] (0.92,-0.24) -- (1.49,-0.17) -- (1.43,0.04);
\draw [thick, dred, rotate around={180:(0.68,0.12)}] (0.42,0.47) --(0.89,0.6) -- (1.14,0.43);

\draw [fill] (0.68,0.12) circle [radius=0.01];
\draw [fill] (1.7,0.29) circle [radius=0.01];
\draw [fill] (0.475,-0.365) circle [radius=0.01];
\draw [fill] (-0.125,0.41) circle [radius=0.01];
\draw [fill, rotate around={180:(0.677,0.12)}] (1.31,0.23) circle [radius=0.01];
\draw [fill, rotate around={180:(0.68,0.12)}] (0.13,0.32) circle [radius=0.01];
\draw [fill, rotate around={180:(0.68,0.12)}] (0.54,-0.21) circle [radius=0.01];
\draw [fill] (1.43,0.04) circle [radius=0.01];
\draw [fill,  rotate around={180:(0.68,0.12)}] (1.43,0.04) circle [radius=0.01];
\draw [fill] (0.43,0.475) circle [radius=0.01];
\draw [fill,  rotate around={180:(0.68,0.12)}] (0.43,0.475) circle [radius=0.01];
\draw [fill] (1.15,0.42) circle [radius=0.01];
\draw [fill,  rotate around={180:(0.68,0.12)}] (1.15,0.42) circle [radius=0.01];

\draw (0.59,0.05) node {$0$};
\draw (1.22,0.5) node {$z^{3,1}$};
\draw (1.63,0.03) node {$z^{3,2}=p$};
\draw[rotate around={180:(0.68,0.12)}] (1.22,0.5) node {$z^{1,2}$};
\draw[rotate around={180:(0.68,0.12)}]  (1.63,0.01) node {$z^{2,1}=-p$};
\draw (1,-0.32) node {$z^{1,1}$};
\draw (0.4,0.55) node {$z^{2,2}$};

\draw[rotate around={180:(0.68,0.12)}] (0.65,-0.13) node {$p^1$};
\draw[rotate around={180:(0.68,0.12)}] (0.2,0.23) node {$p^2$};
\draw[rotate around={180:(0.68,0.12)}] (1.2,0.15) node {$p^3$};
\draw (0.45,-0.44) node {$-sp^1$};
\draw (-0.2,0.45) node {$-sp^2$};
\draw (1.8,0.32) node {$-sp^3$};

\draw (1.48,0.43) node {$K$};
\draw (-0.1,-0.2) node {$-K$};
\end{tikzpicture}
\caption{Construction from Lemma \ref{lem:beforePlanarCase}: 
$K$ and $-K$ (black), $-\frac{1}{s(K)}K$  (orange), 
$K'$ (yellow), segments  $[-sp^i,z^{i,j}]$, $i=1,2,3$ and $j=1,2$ (red).
}
\end{figure}
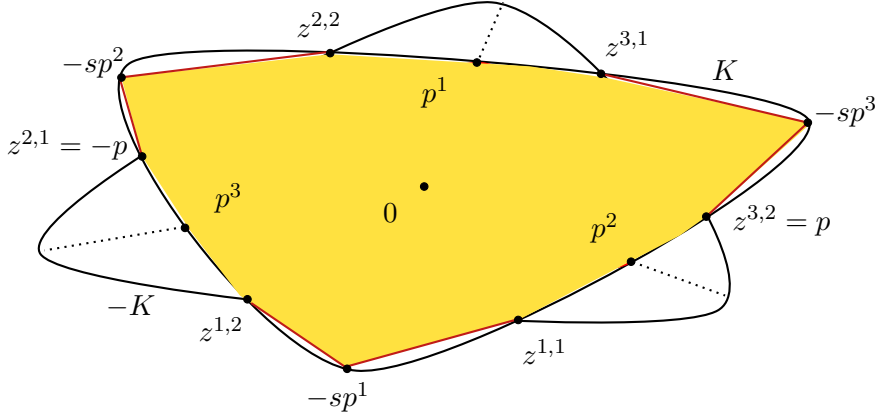 

We now show the properties (i)-(iv) for the set $K'$. 

Obviously, $\pm sp^i \in \bd(\conv(K' \cup (-K')))$, $i=1,2,3$. 
Assume $z \in (\bd(K') \cup \bd(-K')) \setminus ( \{ \pm sp^1, \pm sp^2, \pm sp^3\} \cup \bd(K' \cap (-K')))$, w.l.o.g.~$z \in \bd(K') \cap \inter(\pos(\{-sp^1,z^{1,1}\}))$.
The way we constructed $K'$, this implies $z \in [-sp^1,z^{1,1}]$. Since $-sp^1 \in \bd(\conv(K' \cup (-K')))$ and $z^{1,1} \in \inter(\conv(K' \cup (-K')))$, we obtain $z \in \inter(\conv(K' \cup (-K')))$, a contradiction. Thus, $\conv(K' \cup (-K')))=\conv( \{ \pm sp^1, \pm sp^2, \pm sp^3\} )$.

Moreover, we also have $K' \cap (-K')=K \cap (-K)$, which together with $K' \subset K$ implies $\alpha(K') \geq \alpha(K)$. 

We now need to show that $s(K')=s(K)$. Since $p^i,-sp^i\in K'$, $i=1,2,3$, we only need to check the (non-obvious) fact that $-\frac1s K'\subset K'$. 

Notice that for any $i=1,2,3$ we have
\begin{align*}
    -(K'\cap\pos(\{-z^{i,1},-z^{i,2}\}))&=-K'\cap\pos(\{z^{i,1},z^{i,2}\}) 
= -K\cap\pos(\{z^{i,1}, z^{i,2}\}) \\
&= -(K\cap\pos(\{-z^{i,1},-z^{i,2}\})).
\end{align*}
On the one hand, 
since $K'\subset K$, we have $-\frac1sK'\subset-\frac1sK\subset K$, and therefore
\[
\left(-\frac1sK' \right) \cap\pos(\{-z^{i,1},-z^{i,2}\})\subset K\cap\pos(\{-z^{i,1},-z^{i,2}\})=K'\cap\pos(\{-z^{i,1},-z^{i,2}\})
\]
for every $i=1,2,3$.

On the other hand, we need to show 
\[
(-K')\cap\pos(\{z^{i,1},z^{i,2}\}) \subset K'\cap\pos(\{z^{i,1},z^{i,2}\}), \quad i=1,2,3. 
\]
Clearly, 
\[
(-K')\cap\pos(\{z^{i,1},z^{i,2}\})=(-K)\cap\pos(\{z^{i,1},z^{i,2}\}) \subset K\cap\pos(\{z^{i,1},z^{i,2}\}), \quad i=1,2,3. 
\]
To show the needed inclusion, we additionally prove
\[
(-K')\cap\pos(\{z^{i,1},z^{i,2}\}) \subset \bigcap_{j=1,2}H^\leq_{a^{i,j},\rho_{i,j}}, \quad i=1,2,3. 
\]
Let $H^\leq_{b^{ij},\mu_{i,j}}$ 
be a supporting halfspace of $-K$ in $z^{i,j}$, $i=1,2,3$ and $j=1,2$. By Part (i) of Lemma \ref{lem:crossings}, we know that $-sp^i\in H^>_{b^{ij},\mu_{i,j}}$. Thus, 
\[
H^>_{a^{i,j},\rho_{i,j}}\cap\pos(\{z^{i,1},z^{i,2}\}) \subset H^>_{b^{ij},\mu_{i,j}}\cap\pos(\{z^{i,1},z^{i,2}\}).
\]
We conclude 
\[
\begin{split}
-K'\cap\pos(\{z^{i,1},z^{i,2}\})&=
-K\cap\pos(\{z^{i,1},z^{i,2}\}) \subset 
\bigcap_{j=1,2}H^\leq_{b^{i,j},\mu_{i,j}}\cap\pos(\{z^{i,1},z^{i,2}\}) \\
& \subset
\bigcap_{j=1,2}H^\leq_{a^{i,j},\rho_{i,j}}\cap\pos(\{z^{i,1},z^{i,2}\})
=K'\cap\pos(\{z^{i,1},z^{i,2}\}).
\end{split}
\]
Hence, $-\frac1s K'\cap\pos(\{z^i_1,z^i_2\}) \subset K'\cap\pos(\{z^i_1,z^i_2\})$, $i=1,2,3$ and therefore $-\frac1sK'\subset K'$. 

Since the points $p^1,p^2,p^3$ also build a triple of well-spread asymmetry points of $K'$, we obtain $s(K')=s(K)$.

Now, let $p \in \bd(K' \cap (-K')) \cap \alpha(K') \bd (\conv(K' \cup (-K')))$. 
Within the next paragraph we show by contradiction that $p \in \bd(K') \cap \bd(-K')$. To do so, assume $p \not \in \bd(K') \cap \bd(-K')$. Then, by Lemma \ref{lem:Charact_Opt_Means}, we obtain
$\frac{1}{\alpha(K')} p \not \in \conv (K' \cup (-K')) \setminus (K' \cup (-K'))$, 
and therefore that either $\frac{1}{\alpha(K')} p$ belongs to $\bd(K')$ or $\bd(-K')$. 
Thus, by our cutting-offs of $K$ and $-K$ above, 
\[
\frac{1}{\alpha(K')} p \in \bd(\conv (K' \cup (-K'))) \cap (\bd(K') \cup \bd(-K'))= \{ \pm sp^1, \pm sp^2, \pm sp^3\}.
\]

Since by \eqref{ineq:alpha_s} $\frac{1}{\alpha(K')}<s$, this implies that one of the points $\pm p^1,\pm p^2, \pm p^3$ is not in $K'\cap(-K')\cap\alpha(K')(\bd(\conv(K'\cup(-K'))))$, a contradiction. Thus, $p \in \bd(K') \cap \bd(-K')$,
and hence we may assume w.l.o.g.~in the following that $p:= z^{3,2}$ and $-p:= z^{2,1}$. 

\item
Notice that, because of $0 \in \conv(\{p^1,p^2,p^3\})$ and the way we have chosen the points $-p, p$, we know that $p^1$ is located on the one side, while $p^2, p^3$ are on the other side of $\aff(\{-p,p\})$. Moreover, $p \in \pos\{p^1,p^2\}$ and $-p \in \pos\{p^1,p^3\}$ and again, by Lemma \ref{lem:Charact_Opt_Means}, there exists a pair of halfspaces $H^{\le}_{\pm u, \frac{1}{\alpha(K')} \rho}$ supporting $\conv(K'\cup(-K'))$ in $\pm \frac{1}{\alpha(K')} p$.

Hence, $-sp^2, -sp^3 \in H^{\le}_{u, \frac{1}{\alpha(K')} \rho} \cap  H^{\le}_{-u, \frac{1}{\alpha(K')} \rho}$ and, because $\frac{1}{\alpha(K')} < s$, this implies $p^2, p^3 \in H^{<}_{-u, \rho} \cap H^{<}_{u, \rho}$.
Together with $p \in  H^{=}_{u, \rho}$, 
we obtain the existence of some $d^1 \in \aff (\{-p,p^3\}) \cap \aff (\{p,p^2\})$. Moreover, since $p^1, -sp^2, -sp^3 \in \bd(K')$ with $p^1 \in \pos(\{-sp^2,-sp^3\})$, we have $p^1 \not \in \inter( \conv (\{0,-sp^2,-sp^3\}))$.

Now, we show $-sp^1 \in \conv(\{p^2,p^3,d^1\}) \setminus [p^2,p^3]$.

Since $0 \in \conv(\{p^1,p^2,p^3\})$, we have $-sp^1 \in \pos(\{p^2,p^3\})$, and since
all the points $-p, p, p^2, p^3$ belong to $\bd(K')$, with $p^3 \in \pos(\{-p,-sp^1\})$, and $p^2 \in \pos(\{p,-sp^1\})$, we obtain $-sp^1 \in \conv(\{-p,p,d^1\})$. Finally, $-sp^1 \in \bd(K')$ and $s>1$ imply $-sp^1 \not \in \conv(\{-p,p,p^2,p^3\})$.

Thus, combining the results from above, we obtain
\begin{equation} \label{eq:valid-sit}
-sp^1 \in \conv(\{p^2,p^3,d^1\}) \setminus [p^2,p^3] 
 \quad \text{and} \quad p^1 \not \in \inter( \conv (\{0,-sp^2,-sp^3\})).
\end{equation}
In the following, we say that a well-spread triple of asymmetry points $p^1,p^2,p^3$ presents a \emph{valid situation}, if they satisfy \eqref{eq:valid-sit}. 

Notice that since $K'$ is of the form \eqref{eq:K'}, we also have $[-sp^2,-p],  [p,-sp^3] \subset \bd(K')$.

The next part of the proof describes a three-step transformation of $K'$ into $\bar K$, such that $s(\bar K)=s$, $p \in \bd(\bar K \cap (- \bar K)) \cap \alpha(K') \, \left(\bd(\conv(\bar K \cup (-\bar K))\right)$, and $\bar p^1,p^2,p^3 \in \bd(\bar K) \cap \bd\left(-\frac{1}{s} \bar K \right)$ is a well spread triple of asymmetry points that presents a valid situation.

\begin{figure}[ht]
 \begin{subfigure}[b]{0.47\textwidth}
    \centering
	\begin{tikzpicture}[scale=3]
	\draw [thick, dashed] (-1,0)  -- (1,0); 
	\draw [thick] (-1, -2) -- (-1,1.5); 
	\draw [thick] (1, -2) -- (1,1.5); 
	\draw [thick] (0.8, -2) -- (0.8,1.5); 
	\draw [thick] (-0.8, -2) -- (-0.8,1.5); 
  \draw[->] [thick] (1, 1.3) -- (1.2, 1.3) node[right] {$u$};
	\draw [thick] (-1,0.7) -- (1,0.93); 
	\draw [thick, dashed, gray] (0.15,0.9) -- (-0.23,-1.38); 
	\draw [thick, dashed, gray] (-0.59,-0.53) -- (1,0.93); 
	\draw [thick, dashed, gray] (0.61,-0.42) -- (-1,0.7); 
	\draw [thick]  (-1,0.7)--(1,0.93); 
	\draw [thick]  (-1,0.7)--(-0.8,0)--(-0.08,-1.9)--(0.8,0)--(1,0.93);
    \draw [thick]  (0.6,-0.42)--(-0.59,-0.54);
	
	
	\draw [fill] (0,0) circle [radius=0.02];
	\draw [fill] (-1,0) circle [radius=0.02]; 
	\draw [fill] (1,0) circle [radius=0.02]; 
	\draw [fill] (-0.8,0) circle [radius=0.02]; 
	\draw [fill] (0.8,0) circle [radius=0.02]; 
	\draw [fill] (1,0.93) circle [radius=0.02]; 
	\draw [fill] (-1,0.7) circle [radius=0.02]; 
	\draw [fill] (0.15,0.9) circle [radius=0.02]; 
    \draw [fill] (-0.23,-1.38) circle [radius=0.02]; 
	\draw [fill] (-0.59,-0.54) circle [radius=0.02]; 
	\draw [fill] (0.6,-0.42) circle [radius=0.02]; 
	\draw [fill] (-0.08,-1.9) circle [radius=0.02]; 
	
	\draw (-0.16,-0.07) node {$0$};
	\draw (0.7,0.07) node {$p$};
	\draw (-0.7,0.07) node {$-p$};
	\draw (1.15,0.07) node {$\frac{1}{\alpha}p$};
	\draw (-1.15,0.07) node {$-\frac{1}{\alpha}p$};
	\draw (0.7,-0.4) node {$p^2$};
	\draw (-0.72,-0.5) node {$p^3$};
	\draw (0.1,-1.9) node {$d^1$};
	\draw (0.12,0.95) node {$p^1$};
	\draw (-0.43,-1.5) node {$-sp^1$};
     \draw (-1.17,0.7) node {$-sp^2$};
	\draw (1.15,1.0)  node {$-sp^3$};
\end{tikzpicture}
  \end{subfigure} \hfill
  \begin{subfigure}[b]{0.47\textwidth}
    \centering
 	\begin{tikzpicture}[scale=3]
	\draw [thick, dashed] (-1,0)  -- (1,0); 
	\draw [thick] (-1, -2) -- (-1,1.5); 
	\draw [thick] (1, -2) -- (1,1.5); 
	\draw [thick] (0.8, -2) -- (0.8,1.5); 
	\draw [thick] (-0.8, -2) -- (-0.8,1.5); 

 \draw[->] [thick] (1, 1.3) -- (1.2, 1.3) node[right] {$u$};
	\draw [thick] (-1,0.7) -- (1,0.93); 
	\draw [thick, dashed, gray] (0.03,0.82) -- (-0.08,-1.9); 
	\draw [thick, dashed, gray] (-0.59,-0.53) -- (1,0.93); 
	\draw [thick, dashed, gray] (0.61,-0.42) -- (-1,0.7); 
	\draw [thick]  (-1,0.7)--(1,0.93); 
    \draw [fill=lgold, fill opacity=0.7]  (-1,0.7)--(-0.8,0)--(-0.59,-0.54)--(-0.05,-1.3)--(0.6,-0.42)--(0.8,0)--(1,0.93)--(-1,0.7);

 \draw [thick, dred]  (-1,0.7)--(-0.8,0)--(-0.08,-1.9)--(0.8,0)--(1,0.93)--(-1,0.7);
	
	
	\draw [fill] (0,0) circle [radius=0.02];
	\draw [fill] (-1,0) circle [radius=0.02]; 
	\draw [fill] (1,0) circle [radius=0.02]; 
	\draw [fill] (-0.8,0) circle [radius=0.02]; 
	\draw [fill] (0.8,0) circle [radius=0.02]; 
	\draw [fill] (1,0.93) circle [radius=0.02]; 
	\draw [fill] (-1,0.7) circle [radius=0.02]; 
	\draw [fill] (0.03,0.82) circle [radius=0.02]; 
	\draw [fill] (-0.59,-0.54) circle [radius=0.02]; 
	\draw [fill] (0.6,-0.42) circle [radius=0.02]; 
	\draw [fill] (-0.08,-1.9) circle [radius=0.02]; 
	\draw [fill] (-0.05,-1.3) circle [radius=0.02]; 
 
	\draw (-0.16,-0.07) node {$0$};
	\draw (0.7,0.07) node {$p$};
	\draw (-0.7,0.07) node {$-p$};
	\draw (1.15,0.07) node {$\frac{1}{\alpha}p$};
	\draw (-1.15,0.07) node {$-\frac{1}{\alpha}p$};
	\draw (0.7,-0.4) node {$p^2$};
	\draw (-0.72,-0.5) node {$p^3$};
    \draw (0.12,-1.2) node {$-sp^1$};
	\draw (0.12,-1.9) node {$d^1$};
	\draw (0.1,0.95) node {$p^1$};
	\draw (-1.17,0.7) node {$-sp^2$};
	\draw (1.15,1.0)  node {$-sp^3$};
\end{tikzpicture}
  \end{subfigure}
  \caption{Construction used in the proof of Lemma \ref{lem:beforePlanarCase}b). On the left: the construction before applying the three-step transformation and on the right the set after the transformation (yellow). In red: $K_{\min,s}$ (see definition before Lemma \ref{lem:PlanarCase_equality}).} 
  \label{fig:two-step-trafo}
\end{figure}
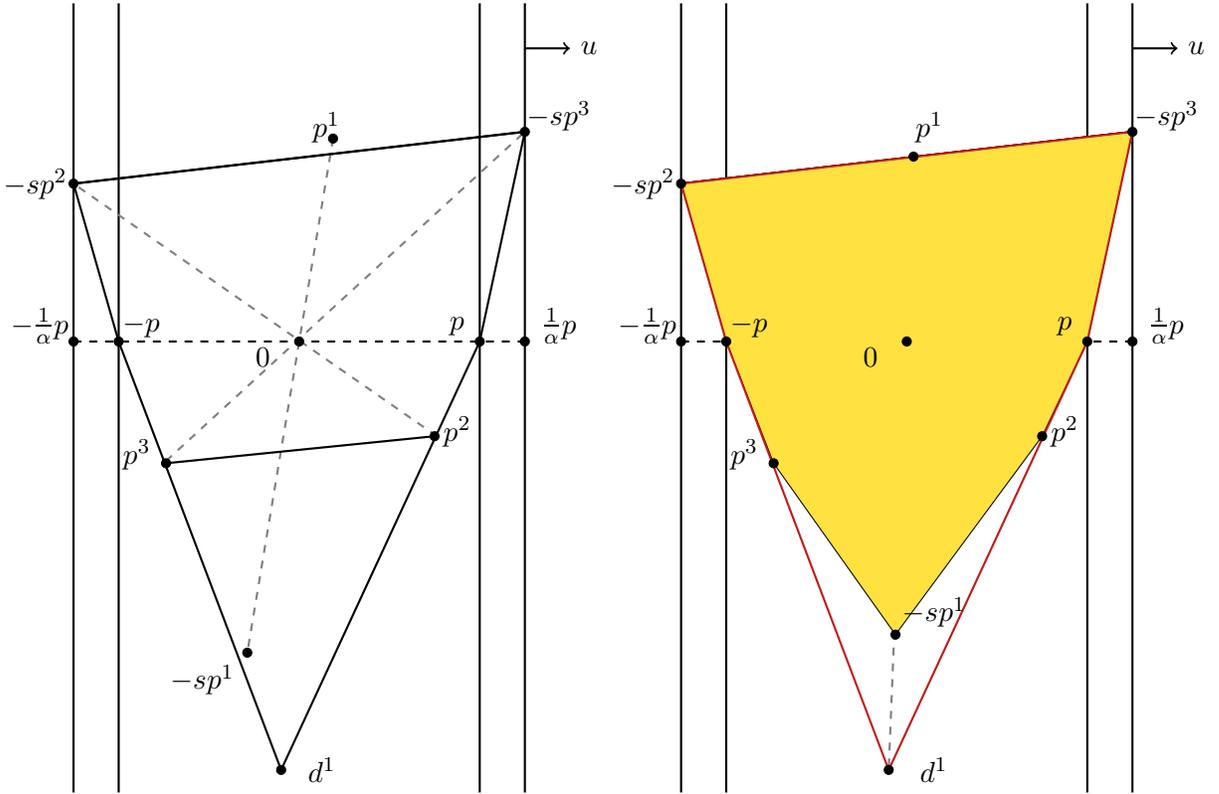

\begin{enumerate}[Step 1]
\item Replace $p^1$ by $\widetilde{p}^1:=\mu p^1$, for some $\mu \le 1$, such that $\widetilde{p}^1 \in[-sp^2,-sp^3]$. 

Obviously, $\widetilde{p}^1 \not \in \inter( \conv (\{0,-sp^2,-sp^3\}))$, and since $\mu<1$ and $s>1$, we obtain directly from $-sp^1 \in \conv(\{p^2,p^3,d^1\})$ that also $-s \widetilde{p}^1 \in \conv(\{p^2,p^3,d^1\}) \setminus [p^2,p^3]$. Hence, the points $\widetilde{p}^1, p^2, p^3$ stay to present a valid situation.

\item 
Replace $\widetilde{p}^1$ by $\bar p^1 := -\gamma d^1\in[-sp^2,-sp^3]$, for some $\gamma > 0$. 

Obviously again, $\bar{p}^1 \not \in \inter( \conv (\{0,-sp^2,-sp^3\}))$. Moreover, since replacing $\widetilde{p}^1$ by $\bar{p}^1$ means moving $-s\widetilde{p}^1$ onto $-s\bar{p}^1$ parallel to $[p^2,p^3]$ such that it belongs to $\lin\{d^1\}$. Thus, $s \gamma d^1 = -s\bar p^1\in \conv(\{d^1,p^2,p^3\}) \setminus [p^2,p^3]$, again. In particular, we obtain $s \gamma \leq 1$, which shows that the points $\bar{p}^1, p^2, p^3$ stay to present a valid situation.
\item 
Replace $K'$ by $\bar K:=\conv(\{\pm p,p^2,p^3,-s \bar p^1,-sp^2,-sp^3\})$. 

The points $\bar p^1 ,p^2, p^3$ still form a well-spread triple of asymmetry points of $\bar K$, which implies also $s(\bar K)=s$. Moreover, we still have $p\in\bd(\bar K\cap(-\bar K))\cap \alpha(K')\left(\bd(\conv(\bar K\cup(-\bar K)))\right)$, which implies that $\alpha(\bar K) \ge \alpha(K')$. 
\end{enumerate}
\end{enumerate}
\end{proof}
Let us explicitly mention, that we do not show (since not needed and possibly false in some cases) that $\alpha(K) = \alpha(\bar K)$ is always true.

\begin{lemma}\label{lem:PlanarCase}
Let $K\in\mathcal K^2$ be Minkowski centered such that $s(K)>\varphi$. Then 
\begin{equation}\label{eq:bound_alpha}
    \alpha(K)\leq \frac{s(K)}{s(K)^2-1}.
\end{equation}
\end{lemma}

\begin{proof}

In order to show \eqref{eq:bound_alpha}, we determine for any fixed $\alpha \in \left[\frac23,1\right]$ below the maximal $s=s(\alpha)$, such that there exists a Minkowski centered $K \in \K^2$  
with $\alpha(K) = \alpha$ and $s(K)=s$. This way we show that $s \le \frac1{2\alpha} +\sqrt{1+\frac1{4\alpha^2}}$, which is equivalent to $\alpha \le \frac{s}{s^2-1}$.

Considering $K'$ to be defined as in Lemma \ref{lem:beforePlanarCase}, we have $\alpha(K) \leq \alpha(K')$ and $s(K)=s(K')$. Hence, if $K'$ fulfills \eqref{eq:bound_alpha}, then 
\[
\alpha(K) \leq \alpha(K') \leq \frac{s(K')}{s(K')^2-1}= \frac{s(K)}{s(K)^2-1}, 
\]
i.e., $K$ also fulfills \eqref{eq:bound_alpha}. Thus, it suffices to show \eqref{eq:bound_alpha} for $K'$. 

In the following, instead of directly showing \eqref{eq:bound_alpha} for $K'$, we simplify the task by considering $\bar K$ as given in Lemma \ref{lem:beforePlanarCase}.
By Lemma \ref{lem:beforePlanarCase} b) we have $s(K')=s(\bar K)$ and $p \in \bd(\bar K \cap (-\bar K)) \cap \alpha(K') (\bd (\conv(\bar K \cup (-\bar K))))$. Thus, we can study the maximal possible value for $s:=s(K)=s(\bar K)$, while keeping $p \in \bd(\bar K \cap (-\bar K)) \cap \alpha(K') \left(\bd(\conv(\bar K \cup (-\bar K))\right)$.


Observe that from the convexity of $\bar K$ 
we directly obtain that $s \gamma  \leq 1$. Thus, we want to characterize the situation, in which $s$ becomes maximal, under the condition $s\gamma\leq 1$. To do this, we compute the explicit value of $\gamma$ in dependence of $s$. 

By Lemma \ref{lem:Charact_Opt_Means}, there exists a pair of halfspaces $H^{\le}_{\pm u, \frac1\alpha \rho}$ supporting $\conv(\bar K\cup(-\bar K))$ in $\pm \frac1\alpha p$. In the following, we assume w.l.o.g.~that
$p=\smallmat{\alpha \\ 0}$ and, since we may choose the above hyperplanes to be orthogonal to $[-p,p]$, $u=\smallmat{1\\0}$ and $\rho = \alpha$.

Remembering that we have $\conv(\bar K\cup(-\bar K)) = \conv(\{\pm sp^1,\pm sp^2, \pm sp^3\})$ and that, following the notations from Lemma \ref{lem:beforePlanarCase}, we assumed $p=z^{3,2}$, $-p=z^{2,1}$, 
we see that $\frac{1}{\alpha} p \in [sp^2,-sp^3]$ and $-\frac{1}{\alpha} p \in [-sp^2,sp^3]$. Hence,
$sp^2 \in H^{\le}_{\smallmat{1 \\ 0} , 1}$ and $sp^3 \in H^{\le}_{\smallmat{-1 \\ 0} , 1}$
and therefore the first coordinates of $p^2$ and $p^3$ are $\frac 1 s$ and $-\frac 1 s$, respectively. Knowing this fact and remembering that $p^2,p^3$ are located on the same side of $\aff\{-p,p\}$,
we may further assume that $p^2=\smallmat{1/s \\ -a}$ and $p^3=\smallmat{-1/s \\ -1}$ for some $a \in (0,1]$.

Let $d^2$ denote the intersection point of $H^{=}_{\smallmat{-1 \\ 0} , 1}$ and $\aff (\{-p,p^3\})$. 
Hence,  $d^2=-p + \mu(-p+p^3) = \smallmat{-\alpha \\ 0} + \mu \smallmat{-\alpha+\frac 1 s \\ 1}$ for some $\mu>0$ and we directly see that $d^2_2 = \mu$. Now, since we have $d^2_1=-1$ we obtain
$d^2_2 = \frac{s(1-\alpha)}{s\alpha-1}$. 
Altogether, 
\[
d^2= \smallmat{-1 \\ \frac{s(1-\alpha)}{s\alpha-1}}.
\]
By definition, $d^2$ should be in $[-\frac 1 \alpha p, -sp^2]$, which implies, because of $-sp^2 = \smallmat{-1 \\ sa}$, that $\frac{s(1-\alpha)}{s\alpha-1} \leq sa$ or, equivalently, $a \geq \frac{1-\alpha}{s\alpha-1}$.

Now, we calculate the coordinates of $d^1$ as the intersection point of the lines $\aff\{p,p^2\}$ and $\aff\{-p,p^3\}$. We obtain that those coordinates satisfy the following system of equations: 
\begin{align*}
d^1_2 &= \frac{s a}{s \alpha-1} d^1_1 - \frac{s a\alpha}{s \alpha-1},\\ 
d^1_2 &= \frac{-s}{s\alpha-1} d^1_1 - \frac{s \alpha}{s\alpha-1}.
\end{align*}

Solving, 
gives us 
\[
d^1=\begin{pmatrix}
\frac{\alpha(a-1)}{a+1} \\ \frac{-2sa\alpha}{(s\alpha-1)(a+1)} 
\end{pmatrix}.
\]
Now, we compute $\gamma$ such that
\[
-\gamma d^1\in[-sp^2,-sp^3]=\left[\begin{pmatrix} -1\\ sa\end{pmatrix},\begin{pmatrix} 1\\ s\end{pmatrix}\right]. 
\]
This means we are looking for some $\lambda\in[0,1]$, such that
\[
-\gamma \begin{pmatrix} \frac{\alpha(a-1)}{a+1}\\ \frac{-2sa\alpha}{(s\alpha-1)(a+1)}\end{pmatrix}
=(1-\lambda) \begin{pmatrix} -1\\ sa\end{pmatrix}+\lambda
\begin{pmatrix} 1\\ s\end{pmatrix}=
\begin{pmatrix} -1+2\lambda\\ s((1-\lambda)a+\lambda)\end{pmatrix},
\]
and it is easy to verify that this implies
\[
\gamma^{-1}=\frac{2\alpha}{(a+1)^2}\left( \frac{2a}{s \alpha-1} - \frac{(a-1)^2}{2}\right).
\]
One may check that, since $a \geq \frac{1-\alpha}{s\alpha-1}$, we have $\gamma^{-1} \geq 0$.


Thus, finding the maximal $s$ with $s\gamma\leq 1$ 
may now be rewritten as
\[
\max s, \text{ such that } s \leq \frac{-2\alpha}{(a+1)^2}\left( \frac{2a}{1-s \alpha}+ \frac{(a-1)^2}{2}\right),
\]
which can be transformed into
\[
s^2+\left(\alpha\left(\frac{a-1}{a+1}\right)^2-\frac 1 \alpha \right)s-1 \le 0.
\]
We are interested in the maximal $s$, i.e., in the larger of the two roots of the quadratic on the left, which is (independently of $\alpha$ and $a$)
\[
s=\frac{1}{2} \left(\frac{1}{\alpha}-\alpha \left(\frac{a-1}{a+1} \right)^2+ \sqrt{4+\left(\alpha\left(\frac{a-1}{a+1} \right)^2-\frac{1}{\alpha}\right)^2} \right)=:s(a,\alpha).
\]
Hence, we obtain the maximal $s$ (in dependence of $\alpha$) from maximizing $s(a,\alpha)$ over $a$, where $\frac{1-\alpha}{s\alpha-1} \le a \le 1$. It is straightforward to verify that $s(\cdot,\alpha)$ is increasing in $(0,1]$. We conclude 
\[
\max\,s=s(1,\alpha)=\frac1{2\alpha} +\sqrt{1+\frac1{4\alpha^2}}.
\]

\end{proof}

For any fixed $s \in (\varphi,2]$ we define 
\begin{align*}
K_{min,s}&:=\conv\left(\left\{\begin{pmatrix} \pm \frac{s}{s^2-1}\\ 0\end{pmatrix}, \begin{pmatrix} 0\\ -s^2\end{pmatrix},\begin{pmatrix} \pm 1 \\ s\end{pmatrix}\right\}\right) \quad \text{and} \\
K_{max,s}&:=\conv\left(\left\{\begin{pmatrix} \pm 1\\ s(s^2-s-1)\end{pmatrix}, \begin{pmatrix} 0\\ -s^2\end{pmatrix},\begin{pmatrix} \pm 1\\ s\end{pmatrix}\right\}\right).
\end{align*}

The following lemma deals with the equality case of Lemma \ref{lem:PlanarCase}. 
\begin{lemma}\label{lem:PlanarCase_equality}
Let $K\in\mathcal K^2$ be Minkowski centered such that $s(K)>\varphi$ and $\alpha(K)=\frac{s(K)}{s(K)^2-1}$. Then there exists a non-singular linear transformation $L$ such that
\[
K_{min,s(K)} \subset L(K) \subset K_{max,s(K)}.
\]
\end{lemma}

\begin{proof}
We know from the proof of the previous lemma that $s:=s(K)>\varphi$ and $\alpha:=\alpha(K)=\frac{s}{s^2-1}$ imply $\alpha\in[\frac23,1]$ and $s = \frac{1+\sqrt{1+4\alpha^2}}{2\alpha}$. Moreover,
equality holds in the maximization process of Lemma \ref{lem:PlanarCase} if and only if $a=1$, $\gamma=\frac{1}{s}$. 
If we use these values in the respective formulas above, we obtain for the set $\bar K$ from Lemma \ref{lem:beforePlanarCase} that
\[
\bar K:=\conv\left(\left\{ \begin{pmatrix} \pm \alpha \\ 0\end{pmatrix},\begin{pmatrix} \pm \frac 1 s \\ -1\end{pmatrix},\begin{pmatrix} 0 \\ -s^2 \end{pmatrix},\begin{pmatrix} \pm 1\\ s\end{pmatrix}\right\}\right). 
\]
However, since $\smallmat{\pm 1/s \\ -1} = \frac1 {s^2} \smallmat{0 \\ -s^2} + (1 - \frac 1{s^2}) \smallmat{ \pm s/(s^2-1) \\ 0}$, we essentially have $\bar K = K_{min,s}$.

Thus, we have shown that the only Minkowski centered convex body with $s(K)>\varphi$ and $\alpha(K)=\frac{s(K)}{s(K)^2-1}$ of the form given in Lemma \ref{lem:beforePlanarCase}b) is $K_{min,s}$.

Next, we show that the only Minkowski centered convex body $K'$ of the form \eqref{eq:K'} with $s(K')>\varphi$ and $\alpha(K')=\frac{s(K')}{s(K')^2-1}$ is still $K_{min,s}$. Notice that before Step 3 of Part (b) in Lemma \ref{lem:beforePlanarCase}, both segments, $[-sp^2,-p]$ and $[-sp^3,p]$ already belong to $\bd(K')$ and that from the previous paragraph the lines $\aff(\{-sp^1,p\})$, $\aff(\{-sp^1,-p\})$, $\aff(\{-sp^2,-sp^3\})$ support $K'$. Thus, $K'$ must equal $K_{min,s}$ also before Step 3. In Step 2, nothing can change (as otherwise $-sp^1$ would lie outside $K'$) and the same holds true in Step 1 (as only for $\mu=1$ we have $-sp^1 \in K'$).

Finally, we investigate the freedom in the design of $K$ before one applies the transformation \eqref{eq:K'}. Let $d^2 = \smallmat{-1 \\ s(s^2-s-1)}$ denote the intersection point of $H^{=}_{\smallmat{-1 \\ 0} , 1}$ and $\aff (\{-p,p^3\})$, as in Lemma \ref{lem:PlanarCase}. Moreover, let $d^3$ be defined, s.t.~ $\{d^3\} = H^{=}_{\smallmat{1 \\ 0} , 1} \cap \aff (\{p,p^2\})$. 
The only possible freedom we have in choosing the original set $K$ is to replace the linear boundary parts $[p,-sp^3]$ and $[-p-sp^2]$, s.t.~$\bd(K) \cap \pos(\{-sp^3,p\}) \subset \conv(\{-sp^3,p, d^3\})$ and $\bd(K) \cap \pos(\{-sp^2,-p\}) \subset \conv(\{-sp^2,-p, d^2\})$, respectively (c.f.~Figure \ref{fig:K_before_transformation}).
However, since we need to ensure that $s(K)=s$, we also have to fulfill 
\begin{align*}
    \bd(K) \cap \pos(\{-sp^3,p\}) &\not \subset \inter( \conv(\{0,-sp^3,p\})) \quad \text{and} \\
    \bd(K) \cap \pos(\{-sp^2,-p\}) &\not \subset \inter( \conv(\{0,-sp^2,-p\})). 
\end{align*}

\begin{figure}[ht]
    \centering
	\begin{tikzpicture}[scale=2.5]
	\draw [thick, dashed] (-1,0)  -- (1,0); 
	\draw [thick] (-1, -2) -- (-1,1.5); 
	\draw [thick] (1, -2) -- (1,1.5); 
	\draw [thick] (0.8, -2) -- (0.8,1.5); 
	\draw [thick] (-0.8, -2) -- (-0.8,1.5); 
  \draw[->] [thick] (1, 1.3) -- (1.2, 1.3) node[right] {$u$};
	\draw [thick, dashed, gray] (-0.62,-0.43) -- (1,0.7); 
	\draw [thick, dashed, gray] (0.63,-0.45) -- (-1,0.7); 
     \draw [thick, dashed, gray] (0,0.7) -- (0,-1.9); 
	\draw [thick]  (-1,0.7)--(1,0.7); 
 \draw [thick, fill=lgold,opacity=0.6]  (-1,0.7)--(-1,0.5)--(0,-1.9)--(1,0.5)--(1,0.7);
 
\draw [thick, dred] (0.8,0) --(1,0.7)--(1.08,0.7)--(0.8,0);
\draw [thick, dred] (-0.8,0) --(-1.08,0.7)--(-1,0.7)--(-0.8,0);

	\draw [fill] (0,0) circle [radius=0.02];
	\draw [fill] (-1,0) circle [radius=0.02]; 
	\draw [fill] (1,0) circle [radius=0.02]; 
	\draw [fill] (-0.8,0) circle [radius=0.02]; 
	\draw [fill] (0.8,0) circle [radius=0.02]; 
	\draw [fill] (1,0.7) circle [radius=0.02]; 
	\draw [fill] (-1,0.7) circle [radius=0.02]; 
	\draw [fill] (1,0.5) circle [radius=0.02]; 
	\draw [fill] (-1,0.5) circle [radius=0.02]; 
    \draw [fill] (0,0.7)  circle [radius=0.02];
	\draw [fill] (-0.62,-0.43) circle [radius=0.02]; 
	\draw [fill] (0.62,-0.43) circle [radius=0.02]; 
	\draw [fill] (0,-1.9) circle [radius=0.02]; 

	\draw (-0.16,-0.07) node {$0$};
	\draw (0.7,0.07) node {$p$};
	\draw (-0.7,0.07) node {$-p$};
	\draw (1.15,0.07) node {$\frac{1}{\alpha}p$};
	\draw (-1.15,0.07) node {$-\frac{1}{\alpha}p$};
	\draw (0,0.9) node {$p^1$};
    \draw (0.72,-0.5) node {$p^2$};
	\draw (-0.72,-0.5) node {$p^3$};
	\draw (0.1,-2) node {$d^1=-sp^1$};
    \draw (-1.17,0.8) node {$-sp^2$};
	\draw (1.15,0.85)  node {$-sp^3$};
    \draw (-1.1,0.45) node {$d^2$};
    \draw (1.15,0.5) node {$d^3$};
\end{tikzpicture}
\caption{Construction used in the proof of Lemma \ref{lem:PlanarCase_equality}, $K_{max,s}$ (yellow).}
    \label{fig:K_before_transformation}
\end{figure}
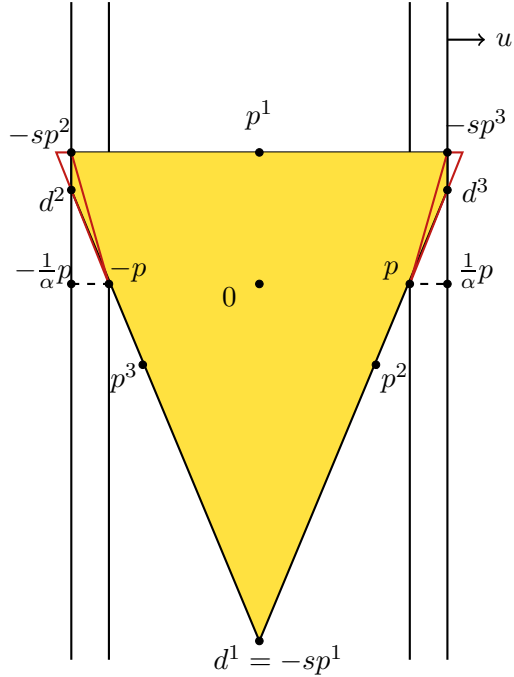

Assuming now, there exists some $x \in \bd(K) \cap \pos(\{-sp^3,p\}) \setminus \conv(\{-sp^3,p, d^3\})$, we would have that $[sp^2,x] \subset \conv(K \cup (-K))$. 
This would imply $x, \frac{1}{\alpha} p, sp^2 \in \bd(\conv(K \cup (-K))$ with $\frac{1}{\alpha} p \in \pos(\{sp^2,x\})$ and therefore $p \in \bd(K)$ as well as 
$\frac{1}{\alpha} p \in \inter(\conv(\{p,x,sp^2\})) \subset \inter(\conv(K \cup (-K)))$ and, since $\pm \frac 1 \alpha p$ were the only points in $\frac 1 \alpha K \cap (-K) \cap \conv(K \cup (-K))$, this would imply $\alpha(K) < \alpha (K') = \alpha$
in contradiction to our assumption that $\alpha(K)=\alpha$. Thus, $\bd(K) \cap \pos(\{-sp^3,p\}) \subset \conv(\{-sp^3,p, d^3\})$. 
Using a similar argument, it follows that $\bd( K) \cap \pos(\{-sp^2,-p\}) \subset \conv(\{-sp^2,-p, d^2\})$.

On the other hand, observe that we may choose $K = \conv\left(\left\{-sp^1,d^2,-sp^2,-sp^3,d^3\right\}\right)$, which equals $K_{max,s}$ up to a linear transformation (such that $K' = K_{min,s}$). All in all, we have shown that 
there always exists some linear transformation $L$ such that $K_{min,s} \subset L(K) \subset K_{max,s}$.

Finally, let us realize that for any $K$ with $s=s(K)$, $s\in(\varphi,2]$, such that $K_{min,s} \subset K \subset K_{max,s}$ we have $\alpha(K)= \frac{s}{s^2-1}$. Notice that in the proof of \cite[Theorem 1.7 a),(ii)]{BDG1} it is shown that $\alpha(K_{max,s})=\frac{s}{s^2-1}$. Since $K_{min,s} \cap (-K_{min,s})=K_{max,s} \cap (-K_{max,s})$ and $\conv(K_{min,s} \cup (-K_{min,s})))=\conv(K_{max,s} \cup (-K_{max,s})))$, we have $\alpha(K_{min,s})=\frac{s}{s^2-1}$, too. Thus, for any $K_{min,s} \subset K \subset K_{max,s}$ holds $\alpha(K)= \frac{s}{s^2-1}$.
\end{proof}

Now, we are ready to prove the main Theorem.

\begin{proof}[Proof of Theorem \ref{thm:PlanarCase}]
In \cite[Theorem 1.7]{BDG1} it is shown that $\frac{2}{s(K)+1} \leq \alpha(K) \leq 1$. 

Combining this with Lemma \ref{lem:PlanarCase}, we obtain 
\[
\frac{2}{s(K)+1} \leq \alpha(K) \leq \min \left\{ 1, \frac{s(K)}{s(K)^2-1} \right\}.
\]
Now, for any given $s \in [1,2]$ consider the set $K_{max,s}$ and notice that $u^1 \in \bd(K_{max,s}) \cap \bd(-K_{max,s}) \cap \bd(\conv(K_{max,s} \cup (-K_{max,s}))$ if $s \leq \varphi$, which shows $\alpha(K_{max,s})=1$. On the other hand, for $s > \varphi$, Lemma \ref{lem:PlanarCase} shows that $\alpha(K_{max,s})=\frac{s(K)}{s(K)^2-1}$. 

We now show that for every $s\in[1,2]$ and $\alpha\in\left[\frac{2}{s+1},\min\{1,\frac{s}{s^2-1}\}\right]$ there exists $K_{s,\alpha}\in\mathcal K^2$, such that $s(K_{s,\alpha})=s$ and $\alpha(K_{s,\alpha})=\alpha$. To do so, let $S=\conv(\{p^1,p^2,p^3\})$ be a regular Minkowski centered triangle and $K_s=S\cap(-sS)$, $s\in[1,2]$. By \cite[Remark 4.1]{BDG1} we have that $K_s$ is Minkowski centered with $s(K_s)=s$ and $\alpha(K_s)=\frac{2}{s+1}$. 
Moreover, defining $q^i$, $i=2,3$, to be the vertices of $K_s$, which are the intersection point of the edges with the normal vectors $\frac{s}{2}p^i$ and $-\frac12p^1$, respectively, 
we see that $-\frac12 p^1,-\frac{1}{s} q^2,-\frac 1s q^3$ is a well-spread triple of asymmetry points of $K_s$.

The idea is to define a continuous transformation $f:\{K_s:s\in[1,2]\}\times[0,1]\rightarrow\mathcal K^2$ with  $s(f(K_s,t))=s$ for all $t\in[0,1]$, while $f(K_s,0) = K_s$ and $\alpha(f(K_s,1)) = \min\left\{ 1, \frac{s(K)}{s(K)^2-1} \right\}$. 

\begin{figure}[ht]
\centering
    \begin{tikzpicture}[scale=2.7]
    \draw [thick, dgreen, fill opacity=0.7] (-0.11547,0.8) -- (0.11547,0.8) -- (0.75038,-0.2997)-- (0.6348,-0.5)-- (-0.6348,-0.5)-- (-0.75038,-0.2997)--(-0.11547,0.8);

    \draw [thick, fill=lgold, fill opacity=0.7] (-0.11547,0.8) -- (0.11547,0.8) -- (0.6348,-0.1)-- (0.6348,-0.5)-- (-0.6348,-0.5)-- (-0.6348,-0.1)--(-0.11547,0.8);
    \draw [thick, dred] (0,0.8) -- (0.6348,0.03)-- (0.6348,-0.5)-- (-0.6348,-0.5)-- (-0.6348,0.03)--(0,0.8);
    \draw [thick, gray] (0.11547,-0.8) -- (-0.11547,-0.8) -- (-0.75038,0.2997)-- (-0.6348,0.5)-- (0.6348,0.5)-- (0.75038,0.2997)--(0.11547,-0.8);
    \draw [thick, black, dotted] (-0.11547,0.8) -- (0,1)--(0.11547,0.8); 
    \draw [thick, black, dotted] (0.75038,-0.2997)-- (0.87,-0.5)--(0.6348,-0.5);
    \draw [thick, black, dotted] (-0.75038,-0.2997)-- (-0.87,-0.5)--(-0.6348,-0.5);
    \draw [thick, black, dotted] (-0.11547,-0.8) -- (0,-1)--(0.11547,-0.8);
    \draw [thick, black, dotted] (0.75038,0.2997)-- (0.87,0.5)--(0.6348,0.5);
    \draw [thick, black, dotted] (-0.75038,0.2997)-- (-0.87,0.5)--(-0.6348,0.5);

    \draw [thick, gray, dashed] (-0.6348,-0.5) -- (0.405,0.3);
    \draw [thick, gray, dashed] (0.6348,-0.5) -- (-0.405,0.3);
    \draw [thick, gray, dashed] (0,-0.5) -- (0,0.8);

    
    \draw [fill] (0,-0.01) circle [radius=0.01];
    \draw [fill] (0.69,-0.4) circle [radius=0.02];
    \draw [fill] (0.64,-0.5) circle [radius=0.02];
    \draw [fill] (-0.64,-0.5) circle [radius=0.02];
    \draw [fill] (-0.69,-0.4) circle [radius=0.02];
    \draw [fill] (0,0.8) circle [radius=0.02];

    \draw [fill] (0,-0.5) circle [radius=0.03]; 
    \draw [fill] (0.405,0.3) circle [radius=0.03]; 
    \draw [fill] (-0.405,0.3) circle [radius=0.03]; 

    \draw (-0.07,-0.07) node {$0$};
    \draw (0,1.1) node {$p^1$};
    \draw (1,-0.55) node {$p^2$};
    \draw (-1,-0.55) node {$p^3$};
    \draw (0.6,-0.65) node {$q^2$};
    \draw (0.3,0.9) node {$\frac s2 p^1$};
    \draw (0.9,-0.2) node {$\frac s2 p^2$};
    \draw (-0.9,-0.2) node {$\frac s2 p^3$};
    \draw (0.53,0.35) node {$-\frac 1s q^3$};
    \draw (-0.55,0.35) node {$-\frac 1s q^2$};
    \draw (-0.6,-0.65) node {$q^3$};
    \draw (0,-0.65) node {$-\frac 12 p^1$};
     \end{tikzpicture}
\caption{Transformation within the proof of Theorem \ref{thm:PlanarCase}: $-S \cap (s S)$ (gray), $K_s = S \cap (-s S)$ before the transformations (green), the transformed set after Step 1 (filled yellow), the transformed set after Step 2 (red), the asymmetry points $-\frac 12 p^1, -\frac 1s q^2, -\frac 1s q^3$ (big black dots) of $f(K_s,t)$, $t\in[0,1]$, and $S$ as well as $-S$ (dotted). 
}
\label{fig:ScapsS}
\end{figure}
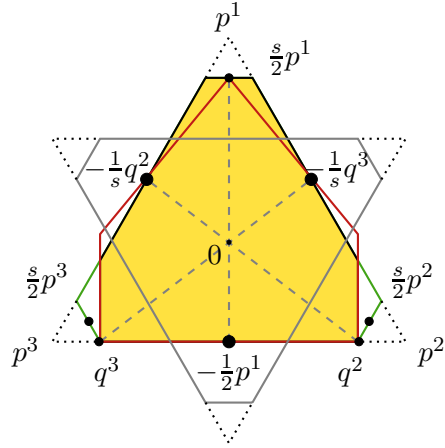

This is done in two steps:
\begin{enumerate}[Step 1]
    \item 
    For $t\in[0,\frac12]$, continuously rotate the lines containing the edges of $K_s$ supporting the point $\frac{s}{2}p^2$ around $q^2$ and $\frac{s}{2}p^3$ around $q^3$, respectively, such that at the end of the transformation, the new edges are both orthogonal to the edge containing 
    $-\frac12 p^1$.
    \item For $t\in[\frac12,1]$, continuously rotate the lines containing the edges of $K_s$, which contain $-\frac 1s q^2$ and $-\frac 1s q^3$, respectively, around those points, s.t.~at the end of the transformation, the new edges intersect in $\frac{s}{2}p^1$. 
\end{enumerate}
For every $u\in\bd(\mathbb B_2)$ let $\rho(u)>0$ be defined such that $\rho(u)u\in\bd(K_s)$. It is very simple to verify that $\rho(u)/\rho(-u)\in[1/s,s]$ for all $u\in\bd(\mathbb B_2)$ after each step. Thus, $-\frac{1}{s}f(K_s,t)\subset f(K_s,t)$, $t\in[0,1]$. Moreover, these transformations are done in a way that the asymmetry points $-\frac12 p^1,-\frac{1}{s} q^2,-\frac1s q^3 \in\bd(f(K_s,t))$, $t\in[0,1]$ are kept to be asymmetry points. 
Hence, by Proposition \ref{prop:Opt_Containment}, $s(f(K_s,t))=s$ for every $s\in[\varphi,2]$. Recognize that $f(K_s,t)$ equals (up to a linear transformation) the corresponding $K_{max,s}$. 
Since the transformation $f$ is continuous, $\alpha(f(K_s,0))=\alpha(K_s)=\frac{2}{s+1}$, and $\alpha(f(K_s,1))=\min\{1,\frac{s}{s^2+1}\}$, we conclude that $\{\alpha(f(K_s,t)):t\in[0,1]\}=\left[\frac{2}{s+1},\min \left\{1,\frac{s}{s^2-1} \right\}\right]$, for every $s\in[1,2]$, as desired.
\end{proof}

\section{Diameter-width-ratio for (pseudo-)complete sets}
For $C_1, \dots, C_k  \in\K^n$ we say $C_1 \subset \ldots \subset C_k$ is \textcolor{dred}{left-to-right optimal}, if $C_1 \subset^{opt} C_k$.

We recall the characterization of pseudo-completeness from \cite{BrG2}.
 \begin{proposition}\label{prop:pseudo-complete-charact}
   Let $K,C \in \K^n$ with $s(C)=1$. Then the following are equivalent:
   \begin{enumerate}[(i)]
   \item $K$ is pseudo-complete \wrt~$C$,
   \item $(s(K)+1) r(K,C)= r(K,C)+R(K,C) = \frac{s(K)+1}{s(K)} R(K,C)=  D(K,C)$, and
   \item for every incenter $c$ of $K$ we have
     \[\frac{s(K)+1}{2s(K)}(-(K-c)) \subset \frac{K-K}2 \subset \frac12 D(K,C) C \subset \frac{s(K)+1}2 (K-c)\]
     is left-to-right optimal, which implies that $c$ is also a circumcenter and a Minkowski center of $K$.
   \end{enumerate}
 \end{proposition}

For $K \in \K^n$ a \textcolor{dred}{regular supporting slab} of $K$ is a pair of opposing supporting hyperplanes of $K$, such that at least one of the two hyperplanes supports $K$ in a smooth boundary point. In case when $K$ is a polytope, the latter means that at least one of the hyperplanes supports $K$ in a whole facet.

In \cite{MoSch} there is a characterization of complete sets using the concept of regular supporting slabs presented.
  
\begin{proposition}\label{prop:regular-slabs+chords}
 Let $K,C \in \K^n$. Then the following are equivalent:
     \begin{enumerate}[(i)]
     \item $K$ is complete w.r.t.~$C$,
     \item $b_s(K,C) = D(K,C)$ for all $s$ such that $s$ is the normal vector of a regular supporting slab of $K$. 
     \end{enumerate}
\end{proposition}

Now, we are ready to prove the general dimension bound on the diameter-width ratio for (pseudo-)complete sets.

Recall that $\tau(K) = R\left(K  \cap (-K), \frac{K-K}{2}\right) = r\left(\frac{K-K}{2}, K  \cap (-K)\right)^{-1}$.


\begin{proof}[Proof of Theorem \ref{thm:results_comp}]
We assume w.l.o.g.~that $K$ is Minkowski centered and $r(K,C)=1$. 
Abbreviating $s:=s(K)$ again, we obtain $D(K,C)=(s+1)r(K,C)=s+1$ and 
\begin{equation*}
 \frac{K-K}2 \subset \frac{D(K,C)}{2} C = \frac{s+1}{2} C \subset \frac{s+1}2 (K \cap (-K))
\end{equation*}
from Proposition \ref{prop:pseudo-complete-charact}.
Thus, $C \subset K \cap (-K)$, which implies $w(K,C) \geq w(K,K \cap (-K))$ and since $w(K,K \cap (-K)) = w\left(\frac{K-K}{2},K \cap (-K)\right) = 2r\left(\frac{K-K}{2},K \cap (-K)\right)$ (see \cite{Sch} for basic properties of the width)
\begin{equation} \label{eq:Dw_chain}
    \frac{D(K,C)}{w(K,C)} =\frac{s+1}{w(K,C)} \leq \frac{s+1}{2 r\left(\frac{K-K}2, K \cap (-K)\right)} 
    = \frac{s+1}2 \tau(K) \le \frac{s+1}2 . 
\end{equation}
Moreover, remember that $\alpha(K)=1$ if and only if $\tau(K)=1$ by Proposition \ref{prop:old_results} (ii). 

Now, consider first the case $n$ odd. Let $S=\conv(\{p^1,\dots,p^{n+1}\})$ be a regular Minkowski centered simplex, and for any $s \in [1,n]$ we define  the sets $K = S \cap (-sS)$ and $C=S \cap (-S)$. Then $s(K)=s$, $K \cap (-K) = S \cap (-S) =C$ (see \cite[Remark 4.1]{BDG1}) 
and since $\frac{K-K}{2} \subset^{opt} \frac{s+1}{2} (K \cap (-K))$ (see \cite[Theorem 1.3]{BDG1}), we have
\[
D(K,C)= 2R\left(\frac{K-K}{2},C\right)=2R\left(\frac{K-K}{2},K \cap (-K)\right)= s+1.
\]

Since all facets of $K$ are facets of $S$ or $-sS$, the normal vectors of the facets of $K$ and $C$ are exactly $\pm p^i$, $i \in \{1, \dots, n+1\}$ and therefore all the regular supporting slabs of $K$ have those normal vectors. 
Now, since $\frac{s}{2} p^i, -\frac{1}{2} p^i \in \bd(K)$, while $\pm \frac{1}{2} p^i \in \bd(C)$ it follows that $b_{p^i}(K,C) = s+1= D(K,C)$ for all $i \in \{1,\dots,n\}$, and therefore the completeness of $K$ w.r.t.~$C$ by Proposition \ref{prop:regular-slabs+chords}.

Moreover, from \cite{BDG1}, we know $S \cap (-S) \subset \frac{S-S}{2} \subset \conv( S \cup (-S))$ is left-to-right optimal for odd $n$.
Using $K \cap (-K) = S \cap (-S)$ we obtain  
\[
K \cap (-K) \subset^{opt} \frac{S-S}{2}.  
\]
and since $K \subset S$ implies $\frac{K-K}{2} \subset \frac{S-S}{2}$ 
we conclude that 
\[
K \cap (-K) \subset^{opt} \frac{K-K}{2}, 
\]
i.e., $\tau(K) =1$.

Hence, we see that, with this choice of $K$ and $C=K \cap (-K)$, we have equality all through the inequality chain \eqref{eq:Dw_chain} for all $s \in [1,n]$.

Finally, for even $n$, let $K':= K \times [-1,1]$ and $C':=C\times \left[-\frac{2}{s+1},\frac{2}{s+1} \right]$ with $K,C \in \K^{n-1}$ as above. Then, $K',C' \in \K^n$, and we easily see that $s(K') \le n-1$.  

By \cite[Theorem 1.3]{BDG1}, 
\[
\frac{K-K}{2} \subset^{opt} \frac{s+1}{2} (K \cap (-K))= \frac{s+1}{2} C.
\]
Moreover, $\pm u^{n} \in \bd\left(\frac{K'-K'}{2}\right) \cap \bd\left(\frac{s+1}{2} C'\right)$. Hence, $\frac{K'-K'}{2} \subset^{opt} \frac{s+1}{2} C'$.
Thus, 
\[
\frac{D(K',C')}{2}= R\left(\frac{K'-K'}{2}, C'\right)= \frac{s+1}{2}. 
\]
Notice, that the set of all regular supporting slabs of $K'$ consists of those of $K$ and the new additional one in the direction of the standard unit vector $u^{n}$. For all normal vectors $u$ of such regular supporting slabs of $K'$ we have $b_u(K',C')=s+1=D(K',C')$. Hence, $K'$ is complete w.r.t. $C'$ by Proposition \ref{prop:regular-slabs+chords}.
\end{proof}

To prepare the proof of Theorem \ref{thm:Improved_richter} we first compute $\tau(S)$ for even dimensional Minkowski-centered simplices.

\begin{lemma}\label{lem:min_in_AM_S}
Let $n\in\mathbb N$ be even and let $S\in\mathcal K^n$ be an $n$-dimensional Minkowski-centered simplex. Then
\[
    S\cap(-S) \subset^{opt} \frac{n}{n+1}\frac{S-S}{2},
\]
i.e.~$\tau(S) = \frac{n}{n+1}$.
\end{lemma}

\begin{proof}
We may assume w.l.o.g.~that $S=\conv(\{p^1,\dots,p^{n+1}\})$ is regular with $\|p^i\|=1$, $i=1,\dots,n+1$. From \cite[Lemma 3.2(ii)]{BDG1} we know that
\[
p=\frac{2}{n+1}\sum_{i=1}^{\frac n2}p^i+\frac{1}{n+1}p^{n+1}
\]
is a vertex of $S \cap (-S)$.

Since $\sum_{i=1}^{n+1}p^i=0$, 
    \[
    p=\frac{n}{2(n+1)}\left(\frac{2}{n}\sum_{i=1}^{\frac n2}p^i-\frac 2n \sum_{i=\frac{n}{2}+1}^n p^i\right)
    \in \frac{n}{2(n+1)}\bd(S-S)=\bd \left(\frac{n}{n+1} \frac{S-S}{2} \right).
    \]

Moreover, we know from \cite[Prop. 1.1]{BDG1} that $\frac{S-S}{2} \subset \mathrm{conv}(S\cup(-S))$ and from \cite[Lemma 3.2(ii)]{BDG1} that $S\cap(-S) \subset^{opt}\frac{n}{n+1}\mathrm{conv}(S\cup(-S))$. Thus, we can conclude that $S\cap (-S) \subset^{opt} \frac{n}{n+1}\frac{S-S}{2}$, as desired.
\end{proof}

The \emph{Banach-Mazur distance} between $K,C\in\mathcal K^n$ is the quantity 
\[
d_{BM}(K,C)=\inf\{\rho\geq 1:t^1+K\subset L(C) \subset t^2+\rho K,L\in\mathrm{GL(n)},t^1,t^2\in\mathbb R^n\}.
\]
In \cite{Sch1} a stability for the Banach-Mazur distance near the simplex is shown. Using it and the ideas developed in \cite[Thm. 1.5]{BDG1} we can provide an upper bound of $\tau(K)$ for values of $s(K)$ near $n$, i.e.~when $K$ is almost a simplex.

\begin{proposition}\label{prop:tau_near_simplex}
Let $n$ be even and $K\in\mathcal K^n$ be Minkowski centered with $s=s(K)$ fulfilling $n-\frac1n<s<n$. 
Then
\[
\tau(K) \leq \left( \frac{(n-s+1)(s+1)}{1-(n-s)(n+s(n+1))} - n \right) \frac{n}{n+1}.
\]
\end{proposition}
\begin{proof} 
Let $S=\conv(\{p^1,\dots,p^{n+1}\})$ be a regular Minkowski centered $n$-simplex and $\rho=d_{BM}(K,S)$. 
By the definition of the Banach-Mazur distance this means there exists a regular linear transformation $L$, such that $t^1+S\subset L(K)\subset t^2+\rho S$ for some $t^1,t^2 \in\R^n$ and it
is shown in \cite[Theorem 1.5]{BDG1} that 
\[
\rho \leq \rho_*:=\frac{n-s+1}{1-n(n-s)}.
\]
Since obviously $s(L(K))=s(K)$ and $\tau(L(K))= \tau(K)$ by Proposition \ref{prop:old_results}, we may (w.l.o.g.) replace $K$ by $L(K)$, and thus obtain
$t^1+S \subset K \subset t^2+\rho S$. Moreover, it is shown in the Proof of \cite[Theorem 1.5]{BDG1} that for $n-\frac1n<s<n$ holds 
    \[
    \mu(\rho) S \subset t^1+S \subset K \subset t^2+\rho S\subset (\rho+n(\rho-\mu(\rho))) S,
    \]
where $\mu(\rho) = \frac{n+1}{s+1}(1-s(\rho-1))$. 

Now, by Lemma \ref{lem:min_in_AM_S}, we obtain
    \[
    \begin{split}
        K\cap(-K) & \subset (t^2+\rho S)\cap(-t^2-\rho S) \\
        & \subset (\rho+n(\rho-\mu(\rho))) (S\cap(-S)) \\
        & \subset (\rho+n(\rho-\mu(\rho))) \frac{n}{n+1} \frac{S-S}{2} \\
        & \subset \frac{\rho+n(\rho-\mu(\rho))}{\mu(\rho)} \frac{n}{n+1} \frac{K-K}{2},
    \end{split}
    \]
and since one can easily verify that 
$\frac{\rho+n(\rho-\mu(\rho))}{\mu(\rho)}$
is increasing in $\rho$, it directly follows that
\[
K\cap(-K)\subset\frac{n}{n+1}\frac{(\rho_*+n(\rho_*-\mu(\rho_*)))}{\mu(\rho_*)}\frac{K-K}{2}. 
\]
We obtain
\[
K \cap (-K) \subset \psi \frac{n}{n+1} \frac{K-K}{2},\] 
where 
\[
\begin{split}
\psi & =\frac{(n+1)\rho_*}{\frac{n+1}{s+1}(1-s(\rho_*-1))}-n  = \frac{\frac{n-s+1}{1-n(n-s)}}{\frac{1}{s+1}\left(1-s\frac{(n+1)(n-s)}{1-n(n-s)}\right)} -n \\
& = \frac{(n-s+1)(s+1)}{1-(n-s)(n+s(n+1))} - n.
\end{split}
\]
\end{proof}

\begin{proof}[Proof of Theorem \ref{thm:Improved_richter}]
Let $s:=s(K)$ and keep in mind that by definition $\tau(K) \leq 1$. Moreover, one can easily verify that $\left( \frac{(n-s+1)(s+1)}{1-(n-s)(n+s(n+1))} - n \right)\frac{n}{n+1} \leq 1$
implies $s > n-\frac{1}{n}$. Thus, combining the ideas of the proof of Theorem \ref{thm:results_comp} with Proposition \ref{prop:tau_near_simplex}, we obtain
    \[
    \frac{D(K,C)}{w(K,C)} \leq 
    \frac{s+1}{2}\tau(K)
    \leq \min\left\{\frac{s+1}{2},\frac{s+1}{2} \left( \frac{(n-s+1)(s+1)}{1-(n-s)(n+s(n+1))} - n \right) \frac{n}{n+1} \right \}.
    \]
    Note that $\frac{s+1}{2}$ is increasing, whereas $\frac{s+1}{2} \left( \frac{(n-s+1)(s+1)}{1-(n-s)(n+s(n+1))} - n \right) \frac{n}{n+1}$
    is decreasing for $s \in(n-1/n,n)$. Thus, the minimum of both attains its maximum, whenever 
    \[
   \left( \frac{(n-s+1)(s+1)}{1-(n-s)(n+s(n+1))} - n \right) \frac{n}{n+1}=1.
    \]

    Writing everything in terms of $s$ and solving the equation, we obtain
    \[
    \frac{D(K,C)}{w(K,C)} \leq \frac{s_0+1}{2},
    \]
    where 
    \[
s_0=\frac{n^4+n^3+2n^2+\sqrt{n^8+6n^7+17n^6+28n^5+28n^4+12n^3-4n^2-12n-4}}{2(n^3+2n^2+3n+1)}> n-\frac1n.
    \]
\end{proof}

Next, we present the proof for the even tighter diameter-width ratio bound for pseudo-complete sets in the planar case. 
\begin{proof}[Proof of Theorem \ref{thm:results_pscomp}]
 Again, we may assume w.l.o.g.~that $K$ is Minkowski centered and use the abbreviation $s:=s(K)$. 
 It is easy to see that $\conv(K \cup (-K)) \subset \frac{2s}{s+1} \frac{K-K}2$ (c.f.~\cite[Theorem 1]{BDG1}). Thus,
\[
K \cap (-K) \subset \alpha(K) \conv (K \cup (-K)) \subset \alpha(K) \frac{2s}{s+1} \frac{K-K}2,  
\]
which implies $\tau(K)\le \frac{2s}{s+1}\alpha(K)$. Using the inequality chain \eqref{eq:Dw_chain}, we obtain the tighter bound 
\begin{align*}
    \frac{D(K,C)}{w(K,C)} 
    &\leq \frac{s+1}2 \tau (K) 
    \leq \frac{s+1}2 \min \left\{1, \frac{2s}{s+1}  \frac{s}{s^2-1} \right\}  
    = \min \left\{\frac{s+1}2, \frac{s^2}{s^2-1} \right\}.
\end{align*}
Finally, let $\tilde s:=\max_{s \in [1, 2]} \min \left\{\frac{s+1}2, \frac{s^2}{s^2-1} \right\}$. Since $\frac{s+1}2$ is increasing and $\frac{s^2}{s^2-1}$ decreasing in $s \in [1, 2]$, we see that $\tilde s$ is the solution of the equation $\frac{s+1}2 = \frac{s^2}{s^2-1}$ and has the value 
\[
\tilde s=\frac{1}{3} \left(1+ \sqrt[3]{19-3 \sqrt{33} } +\sqrt[3]{19+3 \sqrt{33} } \right) \approx 1.8393.
\] 
Thus, we obtain for all pseudo-complete $K$
\[
\frac{D(K,C)}{w(K,C)}\leq \frac{\tilde s+1}2 \approx 1.42,
\]
independently of the asymmetry of $K$.
\end{proof}

\begin{remark}\label{rem:tau_for_house}
Of course, there exist Minkowski centered $K \in \K^2$ with $\alpha(K) \neq \tau(K)$ \cite[Example 4.3]{BDG1}. Observe, however, that for any $s \in [\varphi,2]$ 
and $K$ such that $K_{min,s} \subset K \subset K_{max,s}$  with $s(K)=s$,
we have $\alpha(K) = \tau(K)$. 

In the proof of Theorem \ref{thm:PlanarCase} we have shown that 
$\alpha(K)=\frac{s}{s^2-1}$. Note that $K \cap (-K)$ is a hexagon, which we can denote by $\conv \{-p, p, q^1, q^2,q^3,q^4\}$, where $p$ is defined as in the proof of Theorem \ref{thm:PlanarCase}. Hence, the touching points of $K \cap (-K)$ and $\bd \left( \tau(K) \, \frac{K-K}{2} \right)$ must be vertices of $K \cap (-K)$. From the proof of \cite[Theorem 1.7 a),(ii)]{BDG1} we also have $p \in \frac{s}{s^2-1} \bd \left( \frac{K-K}{2} \right)$ and $q^i \in \frac{s+1}{2} \bd \left( \frac{K-K}{2} \right)$, $1 \leq i \leq 4$. Since $s \geq \varphi$, 
this implies $\tau(K) =  \min \left\{ \frac{s}{s^2-1}, \frac{2}{s+1} \right\}=\frac{s}{s^2-1}$.
\end{remark}

Observe that if one could show 
\[
\tau(K) \leq \min \left\{ 1, \frac{s(K)}{s(K)^2-1} \right\},
\]
the bound in Theorem \ref{thm:results_pscomp} could be improved to 
\[
\frac{D(K,C)}{w(K,C)} \leq \min \left\{ \frac{s(K)+1}{2}, \frac{s(K)}{2(s(K)-1)} \right\} \leq \frac{D(\GH,\GH\cap(-\GH))}{w(\GH,\GH\cap(-\GH))} = \frac{\varphi+1}2 \approx 1.31.
\]

\begin{remark}\label{rem:fill_diagram_ratio}
We show that for every pair $(\rho,s)$, with $s \in [1,2]$ and $1 \leq \rho \leq \min \left\{ \frac{s+1}{2}, \frac{s}{2(s-1)} \right\}$, there exists some Minkowski centered $K$, s.t.~s(K)=s and a set $C$, s.t.~$K$ is pseudo-complete w.r.t. $C$ and $\frac{D(K,C)}{w(K,C)}=\rho$ (c.f.~Figure \ref{fig1}).

\begin{figure}[ht]
  \begin{tikzpicture}[scale=5]
    \draw[thick, discont] (0.05,0) -- (0.25,0);
    \draw[thick, discont] (0,0.05) -- (0,0.25);
    \draw [thick] (-0.2,0) -- (0.05,0);
    \draw [thick] (0,-0.2) -- (0,0.05);
    \draw[->] [thick] (0.25, 0) -- (1.7, 0) node[right] {$s(K)$};
    \draw[->] [thick] (0, 0.25) -- (0, 0.9);

    \draw [thick, dblue, shift={(-0.5,-0.7)}] (1,1) -- (2,1);
    \draw[thick, dred, domain=1:1.61, smooth, variable=\x, dred, ,shift={(-0.5,-0.7)}]  plot ({\x}, {(\x+1)/2});
    \draw[thick, dred, domain=1.84:2, smooth, variable=\x,,shift={(-0.5,-0.7)}]  plot ({\x}, {(\x^2)/((\x^2-1))});
    \draw[thick, dred, domain=1.61:1.84, smooth, variable=\x, ,shift={(-0.5,-0.7)}]  plot ({\x}, {(\x+1)/2});
    \draw [thick, gray] (1.5,0.3) -- (1.5,0.635);

    \draw [thick, dashed,gray] (0,0.3) -- (0.5,0.3);
    \draw [thick, dashed,gray] (1.5,0) -- (1.5,0.3);
    \draw [thick, dashed,gray] (0.5,0) -- (0.5,0.3);
    \draw [thick, dashed,gray] (0,0.61) -- (1.1,0.61);
    \draw [thick, dashed,gray] (1.34,0) -- (1.34,0.3);
    \draw [thick, dashed,gray] (0,0.72) -- (1.34,0.72);
    \draw [thick, dashed,gray] (1.112,0) -- (1.112,0.3);

    \draw [fill,shift={(-0.5,-0.7)}] (1,1) circle [radius=0.01];
    \draw [fill,shift={(-0.5,-0.7)}] (2,1) circle [radius=0.01];
    \draw [fill,shift={(-0.5,-0.7)}] (1.84,1.42) circle [radius=0.01];
    \draw [fill,shift={(-0.5,-0.7)}] (1.62,1.31) circle [radius=0.01];
    \draw [fill] (1.5,0.635) circle [radius=0.01];
    
    
    \fill [fill=lgold, fill opacity=0.7, domain=1:1.3, variable=\x,shift={(-0.5,-0.7)}] (1,1) -- (1.3,1)-- (1.3,1.15)--(1,1);
    \fill [fill=lgold, fill opacity=0.7, domain=1.3:1.61, variable=\x,shift={(-0.5,-0.7)}] (1.3,1) -- (1.61,1)-- (1.61,1.3)--(1.3,1.15);
    \fill [fill=lgold, fill opacity=0.7, domain=1.8:2, variable=\x,shift={(-0.5,-0.7)}] (1.61,1.3) --(1.84,1.1)--(1.84,1)--(1.61,1);
    \fill [fill=lgold, fill opacity=0.7, domain=1.8:2, variable=\x,shift={(-0.5,-0.7)}]  (1.84,1)--(1.84,1.1)--(1.9,1.05)--(1.9,1);
    \fill [fill=lgold, fill opacity=0.7, domain=1.8:2, variable=\x,shift={(-0.5,-0.7)}]  (1.9,1)--(1.9,1.05)--(1.9995,1);

\draw[thick, dgreen, domain=1.615:2, smooth, variable=\x,shift={(-0.5,-0.7)}]  plot ({\x}, {\x/(2*(\x-1))});
    
    \draw (-0.25,0.9) node {$\frac{D(K,C)}{w(K,C)}$};
    \draw (-0.1,0.3) node {$1$};
    \draw (-0.1,0.62) node {$1.31$};
    \draw (-0.1,0.72) node {$1.42$};
    \draw (0.5,-0.1) node {$1$};
    \draw (1.5,-0.1) node {$2$};
    \draw (1.34,-0.1) node {$1.84$};
    \draw (1.11,-0.1) node {$\varphi$};
     \end{tikzpicture}
   \caption{Region of all possible values for the diameter-width ratio for pseudo complete sets $K$ in dependence of their Minkowski asymmetry $s(K)$: $\frac{D(K,C)}{w(K,C)} \geq 1$ (blue); $\frac{D(K,C)}{w(K,C)} \leq \min \left\{\frac{s(K)+1}{2}, \frac{s(K)^2}{s(K)^2-1} \right\}$ (red). Construction from Remark \ref{rem:fill_diagram_ratio}: 
   $\left\{ \frac{D(K,C_\lambda)}{w(K,C_\lambda)}, 0 \leq \lambda \leq 1 \right\} 
=  \left[1,  \min \left\{ \frac{s(K)+1}{2}, \frac{s(K)}{2(s(K)-1)} \right\}  \right]$ (yellow, with $\frac{s(K)}{2(s(K)-1)}$ in green).
   }\label{fig1}
\end{figure}
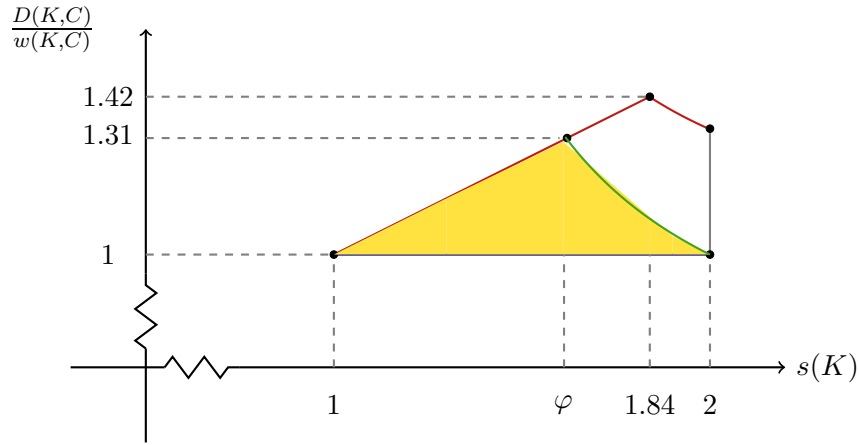

To do so, let $K \in \K^2$ be Minkowski centered with $K:=K_{max,s}$ (where $K_{max,s}$ is defined as in Section 5) and $s:= s(K) \in [1,2]$. Then define $C_\lambda=(1-\lambda) \left( \frac{K-K}{2} \right) + \lambda \frac{s+1}{2} (K \cap (-K))$ with $\lambda \in [0,1]$. This way $C_\lambda$ is a convex combination of $\frac{K-K}2$ and $\frac{s+1}{2}(K \cap(-K))$, and therefore $K \in \K^2_{ps,C_\lambda}$ with $D \left(K,C_\lambda\right)=2$ by Proposition \ref{prop:pseudo-complete-charact}.

By Remark \ref{rem:tau_for_house}, we have $\tau(K) = \min \left\{1,\frac{s}{s^2-1}\right\}$ and 
\begin{align*}
w(K,C_1) &= w \left(\frac{K-K}{2},\frac{s+1}{2} K \cap (-K)\right)= \frac{2}{s+1} w \left(\frac{K-K}{2},\frac{s+1}{2} K \cap (-K)\right) \\
&=  \frac{4}{s+1} r\left(\frac{K-K}{2},\frac{s+1}{2} K \cap (-K)\right) = \frac{4}{s+1} \frac{1}{\tau(K)}. 
\end{align*}
Thus, 
\[
\frac{D(K,C_1)}{w(K,C_1)}= \frac{s+1}{2} \tau(K).
\]
Moreover, $\frac{D(K,C_0)}{w(K,C_0)}=1$. Hence, 
\[
\left\{ \frac{D(K,C_\lambda)}{w(K,C_\lambda)}, 0 \leq \lambda \leq 1 \right\} 
=  \left[1,  \min \left\{ \frac{s+1}{2}, \frac{s}{2(s-1)} \right\}  \right]. 
\]
Note that $\frac{s+1}{2}$ is increasing, while $\frac{s}{2(s-1)}$ is decreasing on $s \in [1,2]$. Thus, $\min \left\{ \frac{s+1}{2}, \frac{s}{2(s-1)} \right\}$ attains its maximum of $\frac{\varphi+1}2 \approx 1.31$, when $s=\varphi$. 
\end{remark}

We conclude the paper with a consideration of the diameter-width ratio of pseudo-complete sets in the euclidean plane. We do this by first recalling the definition of the hood from \cite{BrG} ($\HH_{min}$ there).

The \cemph{hood} may be defined by 
\[
\HH:= \conv \left( \left\{ \begin{pmatrix}
    0 \\ 1
\end{pmatrix}, \begin{pmatrix}
    r \\ \sqrt{1-r^2}
\end{pmatrix}, \begin{pmatrix}
    -r \\ \sqrt{1-r^2}
\end{pmatrix} \right\} \cup \left(r \cdot \B_2 \right) \right),
\]
where
\[
r=\frac{\sqrt{t}}{2}-1+\sqrt{\frac{16}{\sqrt{t}}-t},\quad \text{and} \quad
t=2\left(\frac23\right)^\frac23 \left( (9+\sqrt{69})^\frac13+(9-\sqrt{69})^\frac13 \right).
\]

 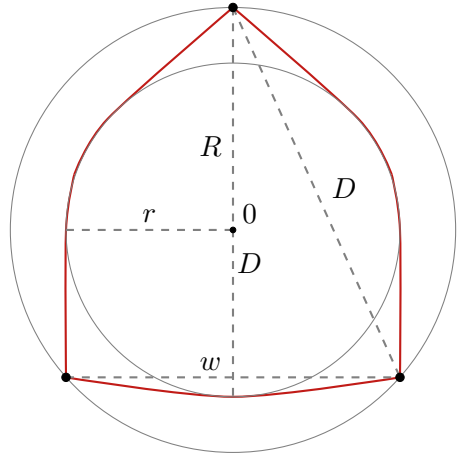
\begin{figure}[ht]
\centering
    \begin{tikzpicture}[scale=3.7]
    \draw [dred, thick] plot [smooth] coordinates {(0.57,0.2) (0.45,0.4) (0,0.8)};
    \draw [dred, thick] plot [smooth] coordinates {(0.57,0.2) (0.6,0) (0.6,-0.53)};
    \draw [dred, thick] plot [smooth] coordinates {(-0.57,0.2) (-0.45,0.4) (0,0.8)};
    \draw [dred, thick] plot [smooth] coordinates {(-0.57,0.2) (-0.6,0) (-0.6,-0.53)};
    \draw [dred, thick] plot [smooth] coordinates {(-0.6,-0.53)  (0,-0.6) (0.6,-0.53) };
    \draw [thick, dashed, gray] (0,0.8) -- (0,-0.6);
    \draw [thick, dashed, gray] (-0.6,0) -- (0,0);
    \draw [thick, dashed, gray] (0,0.8) -- (0.6,-0.53);
    \draw [thick, dashed, gray] (-0.6,-0.53) -- (0.6,-0.53);
    \draw[gray] (0,0) circle [radius=0.6]; 
    \draw[gray] (0,0) circle [radius=0.8];
    \draw [fill] (0,0) circle [radius=0.01]; 
    \draw [fill] (0,0.8) circle [radius=0.015]; 
    \draw [fill] (-0.6,-0.53) circle [radius=0.015]; 
    \draw [fill] (0.6,-0.53) circle [radius=0.015];
    \draw (0.06,0.06) node {$0$};
    \draw (-0.08,0.3) node {$R$};
    \draw (-0.3,0.05) node {$r$};
    \draw (-0.08,-0.48) node {$w$};
    \draw (0.06,-0.12) node {$D$};
    \draw (0.4,0.15) node {$D$};
    \end{tikzpicture}
    \caption{The hood $\HH$ (red); $r(\HH,\B_2) \B_2$ and $R(\HH,\B_2) \B_2$ (gray).
   }\label{fig:hood}
\end{figure}

As shown in \cite{BrG}, the hood has the following properties. First of all, 
\[
r(\HH,\B_2) \B_2 \subset \HH \subset R(\HH,\B_2) \B_2,
\]
with $R(\HH,\B_2)=1$ and $r(\HH,\B_2)=r\approx 0.7935$.

The triangle $\conv \left( \left\{
\begin{pmatrix}
    0 \\ 1
\end{pmatrix}, \begin{pmatrix}
    r \\ \sqrt{1-r^2}
\end{pmatrix}, \begin{pmatrix}
    -r \\ \sqrt{1-r^2}
\end{pmatrix} \right\} \right\} $ is isosceles, 
with the long edges of length $D(\HH,\B_2)=r+1$ and the short of length $w(\HH,\B_2)=2r$. Moreover, we have 
\[
r(\HH,\B_2)+R(\HH,\B_2)=D(\HH,\B_2),
\]
thus $\HH \in \K^2_{ps, \B_2}$ by Proposition \ref{prop:pseudo-complete-charact} and $s(\HH)=\frac{R(\HH,\B_2)}{r(\HH,\B_2)}= \frac{1}{r(\HH,\B_2)} \approx1.27$. Thus,
\[
\frac{D(\HH,\B_2)}{w(\HH,\B_2)}= \frac{r+1}{2r}= \frac{\frac{1}{s(\HH)}+1}{2\frac{1}{s(\HH)}}= \frac{s(\HH)+1}{2} \approx 1.135.
\]

\begin{proof}[Proof of Theorem \ref{lem:dw_eucl}]
Let $K \in \K^2_{ps, \B_2}$, i.e.~ $r(K,\B_2)=D(K,\B_2)-R(K,\B_2)$. W.l.o.g., we assume $R(K,\B_2)=1$. 

Now, on the one hand, since $2r(K,\B_2) \leq w(K,\B_2)$, 
it follows
\[
\frac{D(K,\B_2) }{w(K,\B_2)} \leq \frac{D(K,\B_2)}{2r(K,\B_2)}= \frac{D(K,\B_2)}{2(D(K,\B_2)-R(K,\B_2))}= \left( 2 \left(1- \frac{1}{D(K,\B_2)} \right) \right)^{-1}
\]
However, $2 \left(1- \frac{1}{x} \right)$ is an increasing function for any positive values of $x$, and therefore 
\[
\max_{D(K,\B_2)\in [D(\HH,\B_2), 2]} \frac{D(K,\B_2) }{w(K,\B_2)} \leq 
\left( 2 \left(1- \frac{1}{D(\HH,\B_2)} \right) \right)^{-1}=\frac{D(\HH,\B_2)}{w(\HH,\B_2)}.
\]

On the other hand, since $r(K,\B_2)=D(K,\B_2)-1$, we obtain from \cite[Theorem 3.2]{BrG} 
\begin{align*}
    \frac{D(K, \B_2)}{w(K,\B_2)} & \leq \left(2 \sqrt{1-\left( \frac{D(K, \B_2)}{2 R(K, \B_2)}\right)^2} \cos \left[ \arccos \left( \frac{D(K, \B_2)}{2(D(K, \B_2)- r(K, \B_2))}  \right) \right. \right.\\
    &\left.\left. +\arccos \left( \frac{D(K, \B_2)}{2R(K, \B_2)}  \right)-\arcsin \left( \frac{r(K, \B_2)}{D(K, \B_2)-r(K, \B_2)}  \right) \right]\right)^{-1}\\
    &= \left(2\sqrt{1-\left( \frac{D(K, \B_2)}{2}\right)^2} \cos \left( 2 \arccos \left( \frac{D(K, \B_2)}{2}  \right)-\arcsin \left( r(K, \B_2) \right)\right) \right)^{-1}\\
    &= \left(\sqrt{4- D(K, \B_2)^2} \cos \left( 2 \arccos \left( \frac{D(K, \B_2)}{2}  \right)-\arcsin \left( D(K, \B_2)-1 \right)\right)\right)^{-1}.
\end{align*}
It is easy to verify that $\sqrt{4-x^2} \cos(2\arccos(\frac x 2)-\arcsin(x-1))$ is decreasing for $x \geq \sqrt{3}$.
Thus, 
\[
\max_{D(K,\B_2) \in [\sqrt{3}, D(\HH,\B_2)]} \, \frac{D(K,\B_2) }{w(K,\B_2)} \leq \frac{D(\HH,\B_2)}{w(\HH,\B_2) }.
\]
Since $[\sqrt{3},2]$ covers the full range of possible diameters for pseudo-complete sets with circumradius 1 \cite{BrG} and because $\HH$ attains equality in each of the two inequalities derived above, we conclude 
\[
\frac{D(K,\B_2)}{w(K,\B_2)} \le \frac{D(\HH,\B_2)}{w(\HH,\B_2) } = \frac{s(\HH) +1}2\approx 1.135
\]
for $K \in \K^2_{ps, \B_2}$.
\end{proof}


\bigskip

René Brandenberg -- 
Technische Universität München, Department of Mathematics, Germany. 
\textbf{rene.brandenberg@tum.de}

Katherina von Dichter -- 
	Brandenburgische Technische Universität Cottbus-Senftenberg, Department of Mathematics, Germany. 
\textbf{Katherina.vonDichter@b-tu.de}

Bernardo Gonz\'alez Merino -- 
Universidad de Murcia, 
Departamento de Ingenier\'ia y Tecnolog\'ia de Computadores, 30100-Murcia, Spain.
\textbf{bgmerino@um.es}
\vfill\eject

\end{document}